\documentclass[11pt]{article}
\usepackage{graphicx}

\usepackage[utf8]{inputenc}
\usepackage{amsmath,amsfonts,amssymb,amscd,amsthm,txfonts}

\newtheorem{theorem}{Theorem}[section]
\newtheorem{theo}{Theorem}
\newtheorem{proposition}[theorem]{Proposition}
\newtheorem{lemma}[theorem]{Lemma}
\newtheorem{corollary}[theorem]{Corollary}
\newtheorem{remark}{Remark}[section]

\theoremstyle{definition}
\newtheorem{definition}[theorem]{Definition}

\newcommand{\R}{\varmathbb R}
\newcommand{\C}{\varmathbb C}

\newcommand{\Hs}{\mathcal H}
\newcommand{\W}{\mathcal W}
\newcommand{\E}{\mathcal E}
\newcommand{\reff}[1]{(\ref{#1})}
\newcommand{\arct}{\mbox{Arctan }}
\renewcommand{\Re}{\mbox{Re}}

\renewcommand{\Im}{\mbox{Im}}
\newcommand{\Frac}[2]
{\ensuremath
	{
	\scalebox{1.2}{\textsf{$\frac{#1}{#2}$}}
	}
}
\newcommand{\Exp}[1]
{\ensuremath
	{
	\scalebox{1.2}{\textsf{$e^{#1}$}}
	}
}

\title{Large data low regularity scattering results for the wave equation on the Euclidian space}
\author{Anne-Sophie de Suzzoni\footnote{D\'epartement de Math\'ematiques, Universit\'e de Cergy-Pontoise, Site de Saint Martin, 2, av Adolphe Chauvin, 95302 Cergy-Pontoise Cedex, FRANCE, e-mail : \texttt{anne-sophie.de-suzzoni@u-cergy.fr}}}

\begin{document}

\maketitle

\begin{abstract}We will consider the resolution of the $3D$ non linear wave equation under the assumption of spherical symmetry on the euclidian space. For this purpose, we will build a non trivial measure on distributions such that there exists a set of full measurement onto which the flow is globally defined. We will then discuss different properties of the solutions. \end{abstract}

\tableofcontents

\bigskip

\section{Introduction}

The main point of this paper is to show the existence of globally defined solution for the non linear wave equation with localized large initial data displaying low regularity.

\smallskip

Indeed, the lowest regularity one can obtain using what one would call deterministic tools is, for this equation, and under the assumption of spherical data, $H^s\times H^{s-1}$ with $s> \Frac{7}{10}$, see \cite{these}, which means one must be able to differentiate $s$ times one's initial data and still be in $L^2$ to get a globally well-posed problem. With the tools used here, it will be possible to get globally defined solution with initial data which is spatially localized but not in $L^2$.

\smallskip

%Nevertheless, the problem of the non linear wave equation is locally well-posed in $H^s$, with $s>\Frac{1}{2}$. Thus, one can gain some orders of derivative by constraining one's self to a local problem. The idea here is to develop new tools which would allow one to pass from local well-posedness to global well-posedness without having to enforce important regularity onto the initial data.
In order to gain a release on regularity, that is to say, to get global solutions with low-regularity initial data, we will use a probabilistic point of view. Our initial data will be a random variable which, as afore-mentioned, is not in $L^2$, but which satisfies certain properties almost surely. These properties will enable us to extend local well-posedness to global well-posedness. Though, we will have to reduce our problem to both a problem on a compact manifold and on finite dimension before extending it to the non linear wave equation on the Euclidian space.

\smallskip

Nicolas Burq and Nikolay Tzvetkov have built in \cite{wavun,wavde} a measure on a cartesian product of Sobolev spaces $H^\sigma \times H^{\sigma -1}$ with $\sigma<\Frac{1}{2}$ and showed the existence of a set of full measurement such that the problem was globally well-posed for any initial data in this set. However, the support of the space variable $x\in \R^3$ was a compact (the unit ball of $\R^3$), allowing the Laplace-Beltrami operator with Dirichlet boundary conditions (and so the wave equation) to have a discreet spectrum and eigen functions satisfying nice properties regarding their $L^p$-norms. %The first idea will then be to slightly change these results to get global well-posedness in the case mentioned here, that is to say, when $x$ is in $\R^3$.

\smallskip

%The second idea is to use a space-time compactification, called the Penrose transform, to get a new problem confined to a compact, in order to get nice properties of the eigen functions of the operator related to the free evolution.
Our strategy will be to use a space-time compactification, called the Penrose transform, in order to change the non linear wave equation on the space $\R^3$ into a non linear wave equation on the sphere $S^3$. Indeed, the eigen functions of the Laplace-Beltrami operator on the sphere have properties very similar to the ones on the unit ball of $\R^3$. Thus, we shall use the same tools and methods to study the dynamics of the equation on the sphere. Then, we will return to the equation on $\R^3$ and study its large time dynamics, and the order of regularity of the initial data for which the problem is well-posed.

\smallskip

%This paper shall now have two goals : first, the construction of a measure invariant (or at least increasing) under the flow of the non linear wave equation, then, to show the global well posedness of the equation on a set of full measure.
Considering the reduced problem, we will build a measure on the complex Sobolev space $H^\sigma$, let us call it $\rho$, such that the non linear wave equation is well-posed for all $T\in [-\pi,\pi]$ $\rho$ almost surely, and show this well-posedness. The inverse of the Penrose transform will provide a measure on a space of radial functions of $\R^3$ very similar to $L^2\times H^{-2}$ (where the Sobolev spaces are real and not complex) such that the non linear wave equation is globally well-posed almost surely.

\smallskip

We will also determine the spaces which the initial data belongs to and investigate on the regularity of the solution. As we will see, it is almost surely (with regards to the above mentioned measure) not in $L^2$. Nonetheless, it remains localized initial data and keeps a certain regularity, only not an $L^2$ or an $H^s$ one. Furthermore, we do not need to assume that their existing norms are taken small. Hence, they are said `large' initial data.

\smallskip

Finally, there is a scattering result on the solutions that will be stated. Indeed, when $t\rightarrow \pm \infty$, the global solutions behave as a free evolution (that is to say a linear one) with a suitable initial data.

\medskip

Let us now give in more minute details the reached results. The equation studied here is : 

\begin{equation}\label{nlwe} \left \lbrace{\begin{tabular}{ll}
$\partial_t^2 f - \Delta_{\R^3} f + |f|^\alpha f = 0$ & $(t,x) \in \R\times \R^3$ \\
$f|_{t=0} = f_0$ & $\partial_t f|_{t=0} = f_1$ 
\end{tabular} } \right.
\end{equation}
where $f$ is real and $\alpha \in [2,3[$. The function $f$ is also radial,  namely when using the spherical coordinates $(r,\omega)\in \R^+\times S^2$, $x=r\omega$, $f$ depends only on $r$. We will transform \reff{nlwe} into

\begin{equation}\label{nlhe} \left \lbrace{\begin{tabular}{lll}
$i\partial_T u + \sqrt{1-\Delta_{S^3}} (u) + (1-\Delta_{S^3})^{-1/2}(\cos T +\cos R)^{\alpha-2}|\Re u|^\alpha \Re u = 0$ \\
$(T,R) \in \lbrace (T,R) \in [-\pi,\pi]\times [0,\pi] \; |\;\cos T + \cos R > 0 \rbrace$ \\
$u|_{T=0} = u_0$
\end{tabular} } \right.
\end{equation}
thanks to the Penrose transform :

\begin{equation}\label{penr} \left \lbrace{ \begin{tabular}{ll}
$T= \arct (t+r) + \arct (t-r)$  & $\in ]-\pi,\pi[$ \\
$R = \arct (t+r) - \arct (t-r) $ & $\in [0,\pi[$ \end{tabular}  } \right. \end{equation}

In \reff{nlhe}, $\Delta_{S^3}$ is the Laplace-Beltrami operator on $S^3$. The sphere $S^3$ is parametrized by $(\omega \sin R, \\ \cos R)$ with $R\in [0,\pi]$, and $\omega \in S^2$. The vector $\omega$ is the spherical coordinate on which $f$ did not depend and which was not modified by the Penrose tranform. The function $u_0$ is entirely defined by $f_0$ and $f_1$ (and the reverse is true).

\medskip

The measure we will introduce is defined thanks to the zonal eigen functions of $1-\Delta_{S^3}$, noted $(e_n)_n$, where $\Delta_{S^3}$ is the Laplace-Beltrami operator on the sphere $S^3$. We set $\Omega,P$ a probability space and $(g_n)_n$ a sequence of independant complex gaussian variables of law $\mathcal N(0,1)$. The map

$$\varphi : \left \lbrace{ \begin{tabular}{ll}
$\Omega \rightarrow \mathcal D(S^3) $ \\
$\omega \mapsto \sum_n \Frac{\sqrt 2 g_n(\omega)}{n}e_n$
\end{tabular} } \right. $$
defines an image-measure $\mu$ of $P$ on the zonal distributions of $S^3$. The term `zonal' means that the distribution depends only on the distance to a pole of the sphere, data entirely given by the angle $R\in [0,\pi]$ into he parametrization $(\omega \sin R, \cos R)$.

\smallskip

Moreover, the integral

$$\frac{1}{\alpha+2}\int_{0}^\pi (1+\cos R)^{\alpha-2}|\Re u|^{\alpha+2}\sin^2 RdR$$
is $\mu$ almost surely finite, which allow us to define a non-zero measure :

$$d\rho(u) = \exp \left( -\frac{1}{\alpha+2}\int_{0}^\pi (1+\cos R)^{\alpha-2}|\Re u|^{\alpha+2}\sin^2 RdR \right) d\mu(u)\; .$$

The problem \reff{nlhe} is $\rho$ almost surely well-posed for all $T\in [-\pi,\pi]$.

\smallskip

By using the inverse of \reff{penr}, we get a measure $\eta$ on the pairs of radial distributions of $\R^3$ thanks to which we can reach global well-posedness on \reff{nlwe}.

\begin{theo} %There exists a set $\Pi$, $\eta$-measurable and of full measure such that the solution of the problem \reff{nlwe} is globally well defined.

%\smallskip

%What's more, for any $f_0,f_1 \in \Pi$, $f_0$ is almost surely in $L^p$ for $p\in ]2,6[$, $f_1$ is almost surely in $W^{-1,p}$ for the same $p$, and almost surely, $f_0$ is not in $L^p$ for $p$ outside of $]2,6[$, in particular, $f_0$ is not in $L^2$.

%\smallskip

%Nonetheless, we get that $f_0(r)$ converges toward $0$ when $r\rightarrow 0$ almost surely.
Regarding the problem \reff{nlwe} on $\R^3$, we get that : \begin{itemize}

\item the non linear wave equation \reff{nlwe} has $\eta$ almost surely a global strong solution in $L^2_{-1}$ , with initial data in $L^2_{-1} \times H^{-2}_{-6}$ where $L^2_{m}$ (resp. $H^{-2}_m$) is the set of radial distributions $f$ such that $\left(\frac{1+r^2}{2}\right)^m f$ belongs to $L^2$ (resp. $H^{-2}$),
\item the initial data $f_0,f_1$ is $\eta$ almost surely in $L^p\times W^{-1,p}$ for all $p\in ]2,6[$,
\item the solution $f$ can be writen $f(t) = L(t)(f_0,f_1) + g(t)$, where $L(t)$ is the flow of the linear wave equation  and $g(t)$  is such that for all $p\in ]2\alpha, 6[$ and almost all $t\in \R$, $\left(\frac{1+r^2}{2}\right)^{1/2-2/p} g(t)$ belongs to $L^p$ and $L(t) (f_0,f_1)$ is $\eta$-almost surely localized.
\item $f_0$ is $\eta$ almost surely not in $L^p$ for $p$ outside $]2,6[$, in particular, it is a.s. not in $L^2$,
\item nonetheless, $f_0$ is a.s. localized, ie $f_0(r)$ a.s. converges towards $0$ when $r\rightarrow \infty$, and even $f_0(r) = O(\frac{1}{r^{1+\nu}})$ a.s. for all $0<\nu<\frac{1}{2}$.

\end{itemize}
\end{theo}

Then, we have some scatterring properties. Let us define by $L(t)$ the flow of the free evolution, namely the flow of $(\partial_t^2 -\Delta_{\R^3}) f = 0$.

\begin{theo} Let $\max (\frac{3}{2}\alpha,4)<q<\Frac{9}{2}$. For any solution $f$ of \reff{nlwe}, there exists $f_\infty$ such that $f(t) - L(t)f_\infty$ is in $L^q$ and its norm converges toward $0$ when $t\rightarrow \infty$.
\end{theo}

\paragraph{Plan of the paper}

\smallskip

In part two, we will introduce some pre-requisite useful for the sequel : first, Sobolev's embedding theorem and Strichartz inequalities, then the Penrose transform (the one to change the problem on $\R^3$ into a problem on $S^3$), and finally, general results about gaussian measures on distribution spaces.

\smallskip

The third part is dedicated to the construction of the measure on the zonal distributions of $S^3$ and the existence of a set of full measure on these distributions into which the transformed problem is well-defined for all $T\in [-\pi,\pi]$. We will get from these results the immediate existence of a measure on the radial distributions of $\R^3$ and a set $\Pi$ of full measure satisfying the same properties provided that the reverse Penrose transform is well explicited.

\smallskip

In part four, the spaces into which $\Pi$ is almost surely included and almost surely disjoint from are studied, using theorems given by A.Ayache and N.Tzvetkov in \cite{outilsci}.

\smallskip

Finally, in part five, we focus on scattering properties. Since the Penrose transform is an application both on time and space, the time $t=\infty$ doesn't correspond to $T=\infty$ (indeed, $T\in [-\pi,\pi]$). We will then have to study the dynamics of $f$ the solution of \reff{nlwe} in a completely different way from the dynamics of the corresponding solution of the reduced problem \reff{nlhe}.

% 
% 
% 
% 
% 
% 
% 
% 
% 
% 
% 
% 
% 
% We'll try here to solve the non linear wave equation using random initial data, the main point being that we observe global solution for initial data almost surely in spaces very different from the ones provided by deterministic theory.
% 
% \smallskip
% 
% We will consider the equation 
% 
% \begin{equation}\label{init} \left \lbrace{
% \begin{tabular}{ll}
% $\partial_t^2 f -\Delta f + |f|^\alpha f = 0$ & $t,x \in \R\times \R^3$ \\
% $f|_{t=0} = f_0 $ & $\partial_t f|_{t=0} = f_1$ \end{tabular} } \right. 
% \end{equation} 
% 
% where $\Delta$ is the Laplace-Beltrami operator on $\R^3$, $f$ is a radial function, and for all $2\leq \alpha < 3$.
% 
% \smallskip
% 
% What we'll show is that there exists a measure $\eta$ on distributions such that we can find a set $\Sigma$, $\rho(\Sigma) = \rho(\mathcal D)$ on which the flow of the equation \reff{init} is globally defined. Then, we'll discuss the nature of these initial data.
% 
% \smallskip
% 
% This paper is divided into three parts. First, we'll introduce some useful tools. Then, by using the Penrose compactification, we'll show the results on the sphere $S^3$. The last part consists in reversing the Penrose transform and study the spaces in which the initial data live.
% 

\section{Preliminaries}

\subsection{Sobolev spaces and Strichartz inequalities}

First, let us consider the Sobolev inequalities on a compact boundary-free manifold. 

\begin{theorem}[Sobolev embedding theorem] Let $M$ be a compact boundary-free manifold of dimension $n$. Let $s\in \R$ and $p\in [2, \infty[$ such that $\Frac{1}{2} = \Frac{1}{p} + \Frac{s}{n}$. The space $H^s$ is continuously embedded into $L^p$. That is to say, there exists a constant $C(p,s)$ such that, for all $f\in H^s$,

$$||f||_{L^p} \leq C ||f||_{H^s}.$$ 

For $p=\infty$, we have that $H^s$ is continuously embedded into $L^\infty$ if $s> \Frac{n}{2}$.

\smallskip

More generally, if $\Frac{1}{q} = \Frac{1}{p}+\Frac{s}{n}$ and $p<\infty$, then there exists $C$ such that

$$||f||_{L^p} \leq C ||(1-\Delta)^{s/2} f||_{L^q}$$
\end{theorem}

One can find the proof in \cite{outilssob}.

\begin{remark}In fact, the first inequality is also true on $\R^N$. We will use this theorem on the sphere $S^3$ and the first inequality on the Euclidian space $\R^3$. \end{remark}

Let us now add a variable of time $t$, and build the operator $S(t) = \Exp{-it\sqrt{-\Delta}}$, where $\Delta$ is the Laplace Beltrami operator on $M$, we obtain a Strichartz inequality, see \cite{kapit}. To reach this goal, we need to introduce the space $X^s_T$.

\begin{definition}
A couple of real numbers $(p,q)$, $2<p\leq \infty$ is said admissible (on dimension $3$) if 

$$\frac{1}{p}+\frac{1}{q} = \frac{1}{2}\; .$$

For all admissible couple, and all time $T$ we define :

$$X^s_T = C^0([-T, T], H^s(M)) \cap L^p([-T,T],L^q(M)) \; , \; p=\frac{2}{s} \; .$$

Remark that the definition of $X^s_T$ implies that we choose $s$ in $[0,1[$.
%and its dual

%$$Y^s_T = L^1([-T,T],H^{-s})+L^{p'}([-T,T],L^{q'})$$
%where $p'$ and $q'$ are the respective conjugate numbers of $p$ and $q$.
\end{definition}

We then have the following property : 

\begin{proposition}[Strichartz inequality]\label{besoin} Let $(p,q)$ an admissible couple and $0<T\leq \pi$. There exists $C$ independant from $T$ such that for all $f\in H^s$, we have : 

$$||S(t)f||_{X^s_T} \leq C ||f||_{H^{s}}.$$

\end{proposition}

\subsection{The Penrose transform}

We will use the Penrose transform to change the problem on $\R\times \R^3$ into a problem on a bounded set included in $]-\pi,\pi[\times S^3$

\begin{definition} For all $t\in \R$ and $r\in \R^+$, we define : 

\begin{equation} \left \lbrace{ \begin{tabular}{ll}
$T= \arct (t+r) + \arct (t-r)$  & $\in ]-\pi,\pi[$ \\
$R = \arct (t+r) - \arct (t-r) $ & $\in [0,\pi[$ \end{tabular}  } \right. \end{equation}

The transform $\R\times \R^+ \times S^2 \rightarrow ]-\pi,\pi[\times [0,\pi[ \times S^2$, $(t,r,\omega) \mapsto (T,R,\omega)$ sends $\R^4$ in a bounded set of $\R^4$.
\end{definition}

\begin{lemma} The inverse of this transformation on the set $\lbrace (T,R) \in ]-\pi,\pi[\times[0,\pi[\; | \;\Omega(T,R) = \cos T + \cos R > 0 \rbrace $ is given by :

$$t = \frac{\sin T}{\Omega} \; \mbox , \; r=\frac{\sin R}{\Omega}\; .$$
\end{lemma}

%Proof. Indeed,

%\smallskip

%\begin{tabular}{llll}
%$\cos T = \cos(\arct (t+r))\cos(\arct (t-r)) - \sin(\arct (t+r))\sin(\arct (t-r))$ \\
%$\cos R = \cos(\arct (t+r))\cos(\arct (t-r)) + \sin(\arct (t+r))\sin(\arct (t-r))$ \\
%$\sin T = \cos(\arct (t+r))\sin(\arct (t-r)) + \sin(\arct (t+r))\cos(\arct (t-r))$ \\
%$\sin R = -\cos(\arct (t+r))\sin(\arct (t-r)) + \sin(\arct (t+r))\cos(\arct (t-r))$ \\
%\end{tabular} 

%\smallskip

%Hence, in the one hand : 

%$$\Omega = 2\cos(\arct (t+r))\cos(\arct (t-r)) = \frac{2}{\sqrt{(1+(t+r)^2)(1+(t-r)^2)}}$$

%and in the other hand : 

%$$\frac{\sin T}{\Omega} = \frac{1}{2}(\tan(\arct (t-r)) + \tan(\arct(t+r))) = t$$

%$$\frac{\sin R}{\Omega } = \frac{1}{2}(-\tan(\arct (t-r)) + \tan(\arct(t+r))) = r .$$

%End of proof.

What's more, it turns the d'Alembertian on $\R_t\times \R^3_x$ into a d'Alembertian on $]-\pi,\pi[\times S^3$, see \cite{penrose,bonjam}.

\begin{proposition} Let $f$ be a distribution on $\R_t\times \R^3_x$. We set $v(T,R,\omega)= f(t,r,\omega)\Omega^{-1}$. We have :

\begin{equation} \Omega^3(\partial_T^2 -\partial_R^2 -\frac{2}{\tan R} \partial_R + 1)v = (\partial_t^2-\partial_r^2 - \frac{2}{r}\partial_r) f .\end{equation}

\end{proposition}

The Laplacian onto the sphere $S^2$ doesn't depend on the considered variables $t,r,T,R$, so we get :

$$\begin{tabular}{lll}
$(\partial_t^2 - \Delta_{\R^3} )f$ & $ = $ & $(\partial_t^2-\partial_r^2 - \Frac{2}{r}\partial_r - \Frac{1}{r^2}\Delta_{S^2}) f $ \\
 & $=$ & $ \Omega^3(\partial_T^2 -\partial_R^2 -\Frac{2}{\tan R} \partial_R + 1)v - \Frac{\Omega^3}{\sin^2 R} \Delta_{S^2} v$ \\
 & $ = $ & $ \Omega^3 (\partial_T^2 + 1- \Delta_{S^3}) v $ \end{tabular}$$

\smallskip

Let now see how \reff{nlwe} is changed. We have : 

$$(\partial_t^2 - \Delta_{\R^3})f + |f|^\alpha f = \Omega^3 (\partial_T^2 + 1-\Delta_{S^3}) v + \Omega^{\alpha+1}|v|^\alpha v$$
and

$$v|_{T=0} = \frac{f_0(\tan(R/2))}{(1+\cos R)} \; \mbox , \;  \partial_T v |_{T=0} = \frac{f_1(\tan (R/2))}{(1+\cos R)^2} $$

The Penrose transform sends $\R^4$ on a set where $\Omega$ is always positive. Indeed, in terms of $r$ and $t$, $\Omega $ is : 

$$\Omega = \frac{2}{\sqrt{(1+(t+r)^2)(1+(t-r)^2)}} \; .$$

Thus, the original problem \reff{nlwe} is transformed by the Penrose transform into : 

\begin{equation}\label{true} \left \lbrace{ \begin{tabular}{ll}
$(\partial_T^2 +1 -\Delta_{S^3}) v + \Omega^{\alpha-2} |v|^\alpha v = 0$ & $T,R \in [-\pi,\pi]\times [0,\pi]\cap \Omega^{-1}(]0,2])$ \\
$v|_{T=0} = v_0 = \Frac{f_0(\tan (R/2))}{1+\cos R}$ & $\partial_T v|_{T=0} = v_1 = \Frac{f_1(\tan(R/2))}{(1+\cos R)^2}$
\end{tabular} } \right. \; .
\end{equation}

What we will do is replace the domain of definition of the problem by $[-\pi,\pi]\times [0,\pi[$, $\Omega = \cos T +\cos R$ by $\widetilde \Omega$ defined  as : 

$$\widetilde \Omega  = \left \lbrace{\begin{tabular}{ll}
$\cos T +\cos R$ & \mbox{when} $\cos T +\cos R > 0$ \\
$0$ & \mbox{otherwise} \end{tabular}} \right. \; .$$

The problem is now given by : 

\begin{equation}\label{pinit} \left \lbrace{
\begin{tabular}{ll}
$(\partial_T^2 + 1-\Delta_{S^3} )v +\widetilde \Omega^{\alpha - 2} |v|^\alpha v = 0$ & $T,R \in [-\pi,\pi]\times [0,\pi[$ \\
$v|_{T=0} = v_0 = \Frac{f_0(\tan (R/2))}{1+\cos R}$ & $\partial_T v|_{T=0} = v_1 = \Frac{f_1(\tan(R/2))}{(1+\cos R)^2}$
\end{tabular} } \right.
\end{equation}

If $v$ is a solution of \reff{pinit} then it follows that its restriction to the  domain where $\Omega$ is strictly positive is a solution of \reff{true}. 

\smallskip

But one can see that there is a singularity in the definition of the initial data at $R= \pi$. Indeed, it corresponds to the singularity $r=\infty$ before taking the transformed problem. Dividing by $1+\cos R$ corresponds to multiplying by $\Frac{1+r^2}{2}$. Hence, the original initial data $f_0, f_1$ is ``more regular" than the tranfsormed one. 

\subsection{Stochastic tools}

Let us define the stochastic objects that we will need.

\begin{definition} Let $\Omega, P$ be a probabilistic space. In what follows, $(g_n)_n$ is a sequence of independant random variables of complex normal distribution. 
\end{definition}

%\begin{proposition} Let $(c_n)\in l^2$ and $\lambda \geq 0$. There exists $C,c$ such that :

%$$P(|\sum g_n c_n| > \lambda) \leq C e^{-c\frac{\lambda^2}{\sum |c_n|^2}}.$$
%\end{proposition}

%In order to prove it, one has to decompose $c_n$ and $g_n$ in their real and imaginary parts. Then we use that :

%$$P(\sum h_n a_n > \lambda) \leq e^{-\frac{\lambda^2}{2\sum a_n^2}}$$
%where the $h_n$ real normal distributions and $(a_n)\in l^2(\R)$, see \cite{outilprob}.

%\smallskip

%Indeed, for all $t>0$,

%\smallskip

%\begin{tabular}{lll}
%$P(\sum h_n a_n > \lambda)$ &$=$& $P(e^{t\sum h_n a_n} > e^{t\lambda})$ \\
 %& $\leq$ & $ e^{-t\lambda} E(e^{t\sum h_n a_n})$ \\
 %& $\leq $ & $ e^{-t\lambda} E(\prod e^{th_n a_n})$ \end{tabular}

%We use then the fact that the $h_n$ are independant from each other : $E(\prod e^{th_n a_n}) = \prod E(e^{ta_n h_n})$; and that  $E(e^{th_n} = e^{t^2/2})$.

%$$P(\sum h_n a_n > \lambda)\leq e^{-t\lambda + \frac{t^2}{2} \sum a_n^2}$$

%By choosing $t=\frac{\lambda}{\sum a_n^2}$, we get : 

%$$P(\sum h_n a_n > \lambda) \leq e^{-\frac{\lambda^2}{2\sum a_n^2}}.$$

%\bigskip

%We deduce from this proposition that : 

\bigskip

We have the following lemma, see \cite{wavun}.

\begin{lemma}\label{poumpoumpoum} Let $1\leq q < \infty$. There exists a constant $C$ independant from $q$ such that for all sequence $(c_n)_n$, we have,

$$||\sum g_n c_n ||_{L^q_\omega} \leq C \sqrt{q \sum |c_n|^2}\; .$$

\end{lemma}

\begin{remark} The same inequality is valid when $g_n$ are replaced by real centered gaussian variables and $c_n$ is replaced by a sequence of real numbers. \end{remark}

\section{Existence of a measure and a set of full measure where the flow of the transformed problem is defined for all $T\in[-\pi,\pi]$}

We are about to prove the existence of a finite measure on the complex Sobolev space $H^\sigma$, for any $\sigma<\Frac{1}{2}$ that admits a set of full measure $\Sigma$ such that for all $u\in \Sigma$, there is a solution of \reff{pinit} for all $T\in [-\pi,\pi]$ with initial data $v_0$ and $v_1$ depending on $u$.

\smallskip

We are first going to consider the Sobolev spaces $H^\sigma$ on the sphere $S^3$ and the eigen functions of $1-\Delta_{S^3}$ on this sphere. Remind that the radial disbributions on $\R^3$ are transformed into what are called zonal distributions on $S^3$, that is to say, distributions that only depend on the angle $R$, or on the distance to a pole of $S^3$. The integration unit becomes then $4\pi \sin^2 R dR$.

\subsection{Norms of the Laplace Beltrami operator's eigen functions}

\begin{proposition} The functions $e_n (R) = \Frac{1}{\sqrt 2 \pi}\Frac{\sin n R}{\sin R}$, $n\geq 1$ form a diagonalization basis of $\Delta_{S^3} -1= \partial_R^2 + \Frac{2}{\tan R} \partial_R -1$ on $L^2(\sin^2 R dR)$ with eigenvalues $-n^2$. \end{proposition}

The proof of this proposition can be found in \cite{fourier}.

%It appears that the $e_n$ form an orthonormal set which is dense in the set of zonal functions in $L^2$. Moreover,

%\medskip

%\begin{tabular}{llll}
%$\partial_R \frac{\sin n R}{\sin R}$ & $ = $ &$\frac{n\cos n R}{\sin R} - \frac{\cos R\sin n R}{\sin^2 R}$\\
%$\frac{1}{\tan R} \partial_R \frac{\sin n R}{\sin R}$ & $=$ &$\frac{n\cos R \cos n R}{\sin^2 R} - \frac{\cos^2 R\sin n R}{\sin^3 R}$ \\
%$\partial^2_R \frac{\sin n R}{\sin R}$ & $=$ & $-n^2 \frac{\sin nR}{\sin R} - \frac{n\cos R\cos nR}{\sin^2 R} + \frac{\sin nR}{\sin R}- \frac{n\cos R\cos n R}{\sin^2 R} + 2 \frac{\cos^2 R \sin nR}{\sin^3 R}$ \\
%$(\partial_R^2+ \frac{2}{\tan R}\partial_R)\frac{\sin n R}{\sin R}$ &$=$& $(-n^2+1) \frac{\sin n R}{\sin R}$\end{tabular}

%\medskip

%hence the result.

\begin{proposition} Let $1\leq p \leq \infty$. There exists $C_p$ such that for all $n$ :

\begin{equation}\label{majoration} ||e_n||_{L^p} \leq \left \lbrace{ \begin{tabular}{lll}
$C_p$ & \mbox{if} $p<3$ \mbox , \\
$C_p \log n$ & \mbox{if} $p=3$ \\
$C_p n^{1-3/p}$ & \mbox{otherwise} \end{tabular} } \right. \end{equation}

\end{proposition}

\begin{proof}

\smallskip

We have to sum $|\sin nR|^p |\sin R|^{2-p}$ on $]0,\pi[$. With a variable change $R\mapsto \pi-R$ on $[\frac{\pi}{2},\pi[$, the integral on $]0,\pi[$ is twice the one on $]0,\pi/2]$. What's more, on $]0,\pi/2]$,  $\Frac{2 R}{\pi}\leq \sin R \leq R$.

\smallskip

For $p<3$, we do $|\sin nR| \leq 1$ and sum $R^{2-p}$ on $[0,\pi/2]$. 

\smallskip

For $p\geq 3$, we have to consider separately the integral on $]0,\frac{1}{n}]$ on which we do $\frac{\sin nR}{\sin R} \leq n$, and the one on $[\frac{1}{n},\pi/2]$ on which we use $|\sin nR|\leq 1$ and $\sin^{2-p} R\leq (\frac{2 R}{\pi})^{2-p}$.

\smallskip

\end{proof}

\bigskip

Now that the $e_n$ are introduced, we can give some definitions.

\begin{definition}We set : \begin{itemize}
\item $E_N$ the finite dimensional complex vector space linearly spanned by $\lbrace e_n\; |\; n=1,\hdots ,N\rbrace$,
\item $\Pi_N$ the projection of $H^\sigma$ on $E_N$,
\item $\chi_S$ a positive function compactly supported, with support included in $[-1,1]$ equal to $1$ on $[-1/2,1/2]$,
\item $S_N$ the operator from $H^\sigma$ to $E_N$ defined by $\chi_S(-\Frac{\Delta}{N^2})$ that is to say the operator that sends $\sum c_n e_n$ on $\sum \chi_S(\Frac{n^2}{N^2})c_ne_n$. \end{itemize}
\end{definition}

\begin{theorem} Let $1\leq p \leq \infty$, $S_N$ is continuous from $L^p$ to $L^p$ and the supremum on $N$ of the norms of  $S_N$ as operators is finite. In other terms, there exists $C$ (independant from $N$) such that for all $N$ and all $f\in L^p$, $||S_N f||_{L^p}\leq C ||f||_{L^p}$. What is more, for all $f\in L^p$, the sequence $S_N f$ converges in norm $L^p$ toward $f$.
\end{theorem}

The proof of this theorem can be found in \cite{outilsstrichartz}.

\subsection{``Hamiltonian" problem}

Let us now replace the notation $\widetilde \Omega$ by merely $\Omega$.

\smallskip

The first step consists in changing the equation obtained after applying the Penrose transform, that is

\begin{equation} \label{peq} \left \lbrace{\begin{tabular}{ll}
$\partial_T^2 v + H^2 v + \Omega^{\alpha -2 }|v|^\alpha v = 0$ & \\
$v|_{T=0} = v_0$ \mbox{,} & $\partial_T v|_{T=0} = v_1$\end{tabular}}\right. \end{equation}
into a ``Hamiltonian" form. The operator $H$ is the strictly positive square root of $1-\Delta_{S^3}$ where, again, $\Delta_{S^3}$ is the Laplace-Beltrami operator restricted to zonal functions.

\smallskip

For this purpose, we write $u = v - i\partial_T H^{-1} v$. Let us remark that since we assume that $f$, and so $v$, are real distributions, we have $v= \Re u$. What is more, $H$ is strictly positive, which makes $H^s v$ a real distribution for all $s\in \R$. Also, $\partial_T$ and $H$ commute. The equation \reff{peq} is equivalent to : 

\begin{equation}\label{hamil} \left \lbrace{ \begin{tabular}{ll}
$i\partial_T u + H u + H^{-1}(\Omega^{\alpha-2}|\Re u|^\alpha \Re u) =0$ \\
$u|_{T=0} = u_0  = v_0 + iH^{-1} v_1$ \end{tabular} } \right. \end{equation}

Let us prove this fact.

\smallskip

Indeed, 

\smallskip

\begin{tabular}{lll}
$i\partial_T u $ & $=$ & $i\partial_T v +\partial_T^2 H^{-1} v$ \\
$Hu$ &  $=$ & $Hv -i\partial_T v $ \\
$i\partial_T u +Hu$ & $=$ &  $\partial_T^2 H^{-1} v + H v = H^{-1}(\partial_T^2 v + H^2 v)=-H^{-1}(\Omega^{\alpha-2}|v|^\alpha v).$
\end{tabular}

\smallskip

Therefore, \reff{peq} reduces to \reff{hamil}.

\smallskip

So now, we are provided with an almost-Hamiltonian formulation on $u$ of the equation \reff{peq}. The energy defined as :  

\begin{equation}\label{energie} \E(T,u) = \frac{1}{2}||Hu||_{L^2}^2 + \frac{1}{\alpha+2}\int \Omega^{\alpha-2}|\Re u|^{\alpha +2} \end{equation}
is ``formally" decreasing under the flow for $T\in [0,\pi]$, increasing for $T\in [-\pi,0]$. Indeed, by differentiating  $\E(T,u(T,.))$, we get, as a formal computation, that: 

$$d_T \E = \int  \Re(\partial_T \overline u (H^2 u +\Omega^{\alpha-2}|\Re u|^\alpha \Re u))\sin^2 R dR -\sin T \frac{\alpha-2}{\alpha +2} \int \Omega^{\alpha-1}|\Re u|^{\alpha+2}\sin^2 R dR $$ 

$$=\Re\left(\int \partial_T \overline u (iH\partial_T u)\sin^2 R dR \right)- \sin T \frac{\alpha-2}{\alpha +2} \int \Omega^{\alpha-1}|\Re u|^{\alpha+2}\sin^2 R dR$$

$$=\Re \left(i \int |H^{1/2} \partial_T u|^2 \sin^2 R dR\right)-\sin T\frac{\alpha-2}{\alpha +2} \int \Omega^{\alpha-1}|\Re u|^{\alpha+2}\sin^2 R dR$$
 
$$=-\sin T\frac{\alpha-2}{\alpha +2} \int \Omega^{\alpha-1}|\Re u|^{\alpha+2}\sin^2 R dR  \; .$$

Hence, we can see that since $\frac{\alpha-2}{\alpha +2} \int \Omega^{\alpha-1}|\Re u|^{\alpha+2}\sin^2 R dR \geq 0$, the energy reaches its maximum at $T=0$.

\subsection{``Hamiltonian" equations and approximation}
 
In order to get a well posed problem, we will approach it with ODEs, and so restrict ourselves to finite dimensions.

\smallskip

This is where we will use the definitions of $E_N$, $\Pi_N$, and $S_N$.

\bigskip

By replacing the energy $\E$ with

$$\E_N(T,u) = \frac{1}{2} ||Hu||_{L^2}^2 + \frac{1}{\alpha+2}\int \Omega^{\alpha-2}|S_N \Re u|^{\alpha+2}$$
the corresponding equation on $u$ becomes : 

$$i\partial_T u + Hu + H^{-1}S_N(\Omega^{\alpha-2}(|S_N (\Re u))|^\alpha S_N(\Re u)) =0 $$

We then consider the following equation : 

\begin{equation}\label{cauchy} \left\lbrace{\begin{tabular}{ll}
$i\partial_T u + Hu + H^{-1}S_N(\Omega^{\alpha-2}(|S_N (\Re u))|^\alpha S_N(\Re u)) =0 $ \\
$u|_{T=T_0} =  u_0\in E_N$ \end{tabular} } \right. \end{equation}

\begin{proposition} The equation \reff{cauchy} admits global strong solution in $E_N$.%can be put under a Cauchy form by writing $u=\sum_{n=1}^N (a_n+ib_n )(t) e_n(R)$. The problem is equivalent to : 

%\begin{equation} \left \lbrace{\begin{tabular}{ll}
%$\dot a_n = -nb_n$ \\
%$\dot b_n = na_n +F(a_1,\hdots, a_N)$ \\
%$a_n|_{T=t_0}=a_n^0$ & $b_n|_{T=0}=b_n^0$ \end{tabular}}\right. \end{equation}
%with $ u_0  = \sum_{n=1}^N (a_n^0 + ib_n^0)e_n$.

%As this is a Cauchy problem, there exists a globally defined flow $\Psi_N(t_0,t)$ on $E_N$ for this equation.

\end{proposition}

\begin{proof} Given the structure of the non linearity in dimension $N$, the local well-posedness is obtained by applying Cauchy-Lipschitz theorem. In order to extend the local result to a global one, we have to consider both the equivalence of norms in finite dimension, and the fact that the energy $\E_N$, which controls the $L^2$ norm of $Hu$ and so the $L^2$ norm of $u$, reaches its maximum at $T=0$ (same computation as for $\E$, but it is not only formal in finite dimension). We will then note $\Psi_N(T_0,T)$ the flow of \reff{cauchy} for all $T_0,T \in [-\pi,\pi]^2$.
\end{proof}

\smallskip

The local well-posedness in low regularity spaces will be proved in the next subsection. Also, we will see that this property is uniform in $N$, that is, we get time of existence and controls on $L^p$ norms independant from $N$ thanks to the uniformity of the norms of $S_N$.

\smallskip

The point of using approximated equations is to allow us to approach the global flow of the wave equation. Moreover, the measure on $E_N$ that it will induce converges towards the measure that we will consider on $H^\sigma$.
 
\subsection{Local well-posedness}
 
We will now prove the local well-posedness of the approached and general form of the pseudo-hamiltonian equations and give some inequalities on the norms of their solutions.

\begin{definition}We set 

$$Y^s_T = L^1([-T,T],H^{-s})+L^{p'}([-T,T],L^{q'})$$
where $p'$ and $q'$ are the respective conjugate numbers of $p=\Frac{2}{s}$ and $q$ such that $\Frac{1}{p}+\Frac{1}{q}=\Frac{1}{2}$.
\end{definition}

\begin{remark}The set $Y^s_T$ is not the dual of the above-mentioned $X^s_T$ but since for all $1\leq p\leq \infty$, $||f||_{L^p}=\sup \lbrace \int f\overline g  \; \; | \; ||g||_{L^{p'}}\leq 1\rbrace$, we get that if $F$ maps continuously $H^{\sigma}$ to $X_T^s$, then its adjoint is continuous from $Y^s_T$ to the dual of $H^\sigma$, namely, $H^{-\sigma}$.\end{remark}

The following proposition regroups some consequences of the Sctrichartz inequality.

\begin{proposition} For all $0<s<s_1<1$ there exists $C>0$ depending only on $s$ such that for all $T\in ]0, \pi]$, $f\in H^s$, $g\in Y_T^{1-s}$ and $h\in Y_T^{1-s_1}$, we have :

\begin{itemize}

\item $ ||S(t) f||_{X_T^s} \leq C ||f||_{H^s}$,
\item $ ||\int_{0}^tS(t-t')H^{-1} g(t')dt'||_{X_T^s} \leq C||g||_{Y_T^{1-s}}$,
\item $||(1-S_N) \int_{0}^t H^{-1} S(t-t')h(t')dt'||_{X^s_T}\leq  CN^{s-s_1}||h||_{Y_T^{1-s_1}}$ .
\end{itemize}

\end{proposition}

\begin{proof} 

\smallskip

The first inequality is the Strichartz inequality already mentioned in \reff{besoin}. One can find its proof in \cite{kapit}. It is equivalent to the continuity of $S : f\mapsto (t\mapsto S(t)f ) $ from $H^s $ to $X^s_T$ with a constant independant from $T$ (as long as $T$ is taken in a compact of $\R^+$).

\smallskip

We deduce from that that its adjoint : $ S^* g \mapsto \int_{-T}^T S(-t') g(t') dt'$ is continuous from $Y^{1-s}_T$ to $H^{s-1}$ with a constant independant from $T$. Hence, as $H^{-1}$ is continuous from $H^{s-1}$ to $H^{s}$, and $S$ from $H^s$ to $X^s_T$, we get that $S\circ H^{-1} \circ S^*$ is continuous from $Y^{1-s}_T$ to $X^s_T$. Then, we get from M. Christ and A. Kiselev lemma (see \cite{christkiselev}) the continuity of 

$$g\mapsto \int_{0}^tS(t-t')H^{-1} g(t')dt'$$
from $Y^{1-s}_T$ to $X^s_T$ with a constant independant from $T$. That is to say : 

$$||\int_{0}^tS(t-t')H^{-1} g(t')dt'||_{X_T^s} \leq C ||g||_{Y^{1-s}_T} \; .$$

At last, $(1-S_N)$ is continuous from $H^{s_1}$ to $H^s$ and its norm is less than $CN^{s-s_1}$. Indeed,

$$||(1-S_N) \sum c_n e_n ||_{H^s} = \sqrt{\sum (1-\chi_S(\frac{n^2}{N^2}))n^{2s}|c_n|^2} $$

$$\leq C N^{s-s_1} \sqrt{\sum n^{2s_1}|c_n|^2} = CN^{s-s_1}||\sum c_n e_n ||_{H^{s_1}}$$ 

We get then that $S\circ (1-S_N)\circ H^{-1} \circ S^*$ is continuous from $Y^{1-s_1}_T$ to $X_T^s$ and its norm is less than $CN^{s-s_1}$, which leads to the third inequality, using, once again M. Christ and A. Kiselev lemma.

\end{proof}

\begin{proposition} Let $(p,q)$ be such that $\Frac{1}{p}+\Frac{3}{q} = \Frac{3}{2}-s$ and $p', q'$ their respective conjugate numbers, ($p\geq \Frac{2}{s}$). Again, these definitions imply that $s$ belongs to $[0,1[$. Then,

$$||f||_{L^p_t,L^q_x}\leq C ||f||_{X^s_T}\; \mbox{and} \; ||f||_{Y^s_T}\leq C ||f||_{L^{p'}_t,L^{q'}_x} \; .$$
\end{proposition}

\begin{proof}
We get the inequalities for the extremal couples $p=\Frac{2}{s},q=\Frac{2}{1-s}$ by definition of the norm $X^s_T$ and $p=\infty, \Frac{1}{q} = \Frac{1}{2}-\Frac{s}{3}$ thanks to Sobolev embedding theorem. Regarding the other couples, we deduce the result from Hölder inequalities, indeed, if $\Frac{6}{3-2s} \leq q \leq \Frac{2}{1-s}$ then $\Frac{1}{q}$ can be written as $\Frac{1}{q} = \theta \Frac{1-s}{2}+ (1-\theta)\Frac{3-2s}{6}$ with $\theta \in [0,1]$, hence :

$$||f||_{L^q_x} \leq C ||f||_{L^{\frac{2}{1-s}}}^\theta ||f||_{L^{\frac{6}{3-2s}}}^{1-\theta} \; .$$

We deduce that :

$$||f||_{L^p_t,L^q_x} \leq C ||\; ||f||_{L^{\frac{2}{1-s}}}^\theta \;||_{L^p_t} ||\;||f||_{L^{\frac{6}{3-2s}}}^{1-\theta}\;||_{L^\infty_t} = C||f||_{L^{p\theta}_t,L^{\frac{2}{1-s}}}^\theta ||f||_{L^\infty_t,L^{\frac{6}{3-2s}}}^{1-\theta} \; .$$

Since $\Frac{1}{p} = \Frac{3-2s}{2}-\Frac{3}{p}=\theta (\Frac{3-2s}{2}-\Frac{3-3s}{2}) = \theta \Frac{s}{2}$ and so $p\theta = \Frac{2}{s}$, we get : 

$$||f||_{L^p_t,L^q_x} \leq C ||f||_{X^s_T}^\theta ||f||_{X^s_T}^{1-\theta} = C ||f||_{X^s_T}\; .$$

\end{proof}

\begin{proposition} Let $s\in]0,1[$ and $p$ defined as $s=\Frac{3}{2}-\Frac{4}{p}$. (Let us note that $p\in ]\Frac{8}{3},8[$.) There exists $C$ such that for all $T\in [0,\pi]$, we have : 

$$||S(t)f||_{L^p([-T,T]\times S^3)} \leq C ||f||_{H^s}$$

\end{proposition}

\begin{proof}
Let $s'=\Frac{3}{2}-\Frac{6}{p}= s -\Frac{2}{p}$ and $q$ such that $\Frac{1}{q} = \Frac{1}{p}+\Frac{s'}{3} = \Frac{1}{2}-\Frac{1}{p}$. By Sobolev embedding theorem,

$$||S(t)f||_{L^p_x} \leq C ||(1-\Delta)^{s'/2} S(t)f||_{L^q_x}\; \mbox{and so}\; ||S(t) f||_{L^p([-T,T]\times S^3)}\leq C ||S(t)(1-\Delta)^{s'/2}f ||_{L^p_t,L^q_x} \; .$$

Since $\Frac{1}{p}+\Frac{1}{q}=\Frac{1}{2}$, thanks to Strichartz inequality \reff{besoin},

$$||S(t) f||_{L^p([-T,T]\times S^3)} \leq C ||S(t) (1-\Delta)^{s'/2}f||_{X^{2/p}}\leq C ||(1-\Delta)^{s'/2}f||_{H^{2/p}}\; .$$

As $||(1-\Delta)^{s'/2}f||_{H^{2/p}} = ||(1-\Delta)^{s'/2 + 1/p} f||_{L^2} = ||(1-\Delta)^{s/2}f||_{L^2}=||f||_{H^s}$, we get the result.
\end{proof}

We will deduce from this properties the local well-posedness of the equations. In order to do so, we write $F(t,u) = \Omega^{\alpha-2}(t,.)|\Re u|^{\alpha}\Re u$, and we decompose the solutions with initial data $u|_{T=t_0}=u_0$ by writing them as $u(t+t_0)= S(t)u_0 +v(t)\; , \; v_{t=0} = 0$.

\begin{proposition} Let $p\in ]2\alpha,6[$ and $s$ defined as precedently as $s=\Frac{3}{2}-\Frac{4}{p}$. We choose an initial data $u_0$ such that $||S(t) u_0 ||_{L^p([-\pi,\pi]\times S^3)}\leq A$, with $A$ a finite positive constant. The following problems :

$$ \left \lbrace{ \begin{tabular}{ll}
$i\partial_t u - Hu -H^{-1}F(t,u) = 0$ \\
$u|_{t=t_0}=u_0 $\end{tabular} }\right.$$
and

$$ \left \lbrace{ \begin{tabular}{ll}
$i\partial_t u - Hu -S_N H^{-1}F(t,S_N u) = 0$ \\
$u|_{t=t_0}=S_N u_0 $\end{tabular}} \right.$$
have unique local solutions $u$ and $u_N$ on $[t_0-\tau,t_0+\tau]$ with $\tau = c(1+A)^{-\gamma}$ where $c$  and $\gamma$ are constant depending only on $s$ (and in particular, are independant from $A$). The functions $u$ and $u_N$, $N\geq 1$ can be written as $u(t_0+t)=S(t)u_0+v(t)$ and $u_N(t_0+t) = S(t)S_N u_0 + v_N(t)$ with $v$ and $v_N$ in $X^s_\tau $. Furthermore, there exists $C$ such that $||v||_{X^s_\tau},||v_N||_{X^s_\tau}\leq CA$. We deduce immediately from the proposition \reff{besoin} and the periodicity of $S$ that 

$$\sup_{t'\in[-\tau,\tau]} ||S(t)u(t_0+t')||_{L^p([-\pi,\pi]\times S^3)}\leq CA$$
and that if $u_0\in H^\sigma$, for any $\sigma < s$ (and so for any $\sigma < \frac{1}{2}$) then 

$$||u||_{H^\sigma}\leq ||u_0||_{H^\sigma}+CA \; .$$
\end{proposition}

\begin{proof}
We turn the problem on $u$ into a fix point one on $v$ that depends on $u_0$. We are now looking for a $v$ (resp. $v_N$) satisfying :

$$v=K(v) \mbox{ (resp.} v_N = K_N(v_N) \mbox{ )}$$
with 

$$K(v) = -i\int_{0}^t S(t-t') H^{-1} F(t_0+t',S(t')u_0 + v(t')) dt'$$ 
and 

$$K_N (v_N) = -i\int_{0}^t S(t-t') H^{-1} S_N F(t_0+t',S(t')S_N u_0 +S_N v_N)dt'\; .$$

Since the operator norms $L^p\rightarrow L^p$ of the $S_N$, $N\geq 1$ are bounded by a constant independant from $N$, we can do the proof only on $u$.

\smallskip

We have then to apply the fix point theorem to $K$ on $X^s_\tau$ with $\tau $ small enough.

$$||K(v)||_{X^s_\tau} \leq C ||F(t_0+t,S(t)u_0 +v(t))||_{Y^{1-s}_\tau}\leq C ||F(t_0+t,S(t)u_0+v(t))||_{L^{q'}([-\tau,\tau]\times S^3)}$$
with $q$ satisfying $\Frac{1}{q}+\Frac{3}{q} = \Frac{3}{2}-(1-s)$, that is to say $q=\Frac{8}{1+2s}$ and $q'=\Frac{8}{7-2s}=\Frac{2p}{p+2}$. Finally, as $\alpha \geq 2$ and $0\leq \Omega\leq 2$, for all $w$, we have : 

$$||F(t_0+t,w)||_{L^{q'}} \leq ||\Omega^{\alpha-2}||_{L^\infty}||w^{\alpha+1}||_{L^{q'}}\leq C ||w||_{L^{(\alpha+1)q'}}^{\alpha +1},$$
and so

$$||K(v)||_{X^s_\tau}\leq C(||S(t)u_0||_{L^{(\alpha+1)q'}}^{\alpha+1}+||v||_{L^{(\alpha+1)q'}}^{\alpha+1})$$

What is more, $(\alpha+1)q'=(\alpha+1)\Frac{2p}{p+2}<(\alpha+1)\Frac{2p}{2\alpha+2} = p$ so, as the integration is done on compacts with size $\tau$ in time, $||f||_{L^{(\alpha+1)q'}}\leq C \tau^{\delta/(\alpha+1)}||f||_{L^p}$, with $\delta/(\alpha+1)= \Frac{1}{(\alpha+1)q'}-\Frac{1}{p}>0$.

$$||K(v)||_{X^s_\tau}\leq C \tau^\delta (A^{\alpha+1}+ ||v||_{X^s_\tau}^{\alpha+1})$$

For all $C'$, by choosing $\tau \leq (\Frac{C'}{C(1+(C')^{\alpha+1})})^{1/\delta}(1+A)^{-\alpha/\delta}$, if $||v||_{X^s_\tau}\leq C'A$, then $||K(v)||_{X_s^\tau}\leq C'A$. So, the ball $B(0,C'A)$ is stable under $K$.

\smallskip

Let us do the same with $K(v_1)-K(v_2)$. 

$$||K(v_1)-K(v_2)||_{X^s_\tau}\leq C ||F(t_0+t,S(t)u_0+v_1)-F(t_0+t,S(t)u_0+v_2)||_{L^{q'}}$$

Since $|F(t_0+t,v)-F(t_0+t,w)|\leq C \Omega^{\alpha-2}(t_0+t)|v-w|(|v|^\alpha+|w|^\alpha)$ , and thanks to a H\"older inequality, given that $\Frac{1}{q'} = \Frac{1}{(\alpha +1)q'}+\Frac{\alpha}{(1+\alpha)q'}$,

$$||K(v_1)-K(v_2)||_{X^s_\tau}\leq C ||v_1-v_2||_{L^{(\alpha+1)q'}}(2||\;|S(t)u_0|^\alpha||_{L^{(\alpha+1)q'/\alpha}}+||\;|v_1|^\alpha||_{L^{(\alpha+1)q'/\alpha}}+||\;|v_2|^\alpha||_{L^{(\alpha+1)q'/\alpha}})$$

$$||K(v_1)-K(v_2)||_{X^s_\tau}\leq C \tau^\delta(A^\alpha +||v_1||_{X^s_\tau}^\alpha+||v_2||_{X^s_\tau}^\alpha)||v_1-v_2||_{X^s_\tau}$$

For $\tau = c(1+A)^{-\alpha/\delta}$ with $c$ small enough, we get that $K$ is contracting on the ball of $X^s_\tau$ with center $0$ and radius $C'A$, and so we can apply the fix point theorem on this ball, which is also stable under $K$. There exists a unique solution $v$ of the equation, and it is such that $||v||_{X^s_\tau}\leq C' A$.

\end{proof}
 
\subsection{Measure construction}
 
What we would like to do is building a measure of the form $(\exp{-\E(u)})du$, where $du$ is ``morally" speaking a Lebesgue measure.

\begin{definition}
We write $\mu_N$ the image measure on $E_N$ by $\varphi_N : \omega \mapsto \sum_{n=1}^N \Frac{\sqrt 2}{n} g_n(\omega) e_n(.)$. 
\end{definition}

\begin{proposition} We have 

$$d\mu_N (\sum_{n=1}^N (a_n+ib_n) e_n ) =d_N \Exp{- \sum \frac{n^2}{2} (a_n^2+b_n^2)}\prod_{n=1}^N da_n db_n$$
where $d_N$ is a factor such that $\mu_N (E_N)=1$.

\smallskip

What is more, the sequence $(\varphi_N)_N$ is a Cauchy sequence in $L^2(\Omega, H^\sigma_R)$ for all $\sigma<1/2$, hence it converges toward a function $\varphi$ in $L^2(\Omega,H^s_R)$. We note $\mu$ the image measure on $H^\sigma$ by $\varphi$.

\end{proposition}

\begin{proof} The $g_n$ being independant, the vector $(\sqrt 2 g_1,\hdots,\frac{\sqrt 2}{N}g_N)$ is a complex gaussian random variable with average $0$ and covariance matrix  : 

$$\begin{pmatrix}
1/2 & & \\
(0) & \ddots & (0) \\
 & & N^2/2 \end{pmatrix}$$
that is to say a real gaussian random variable of covariance matrix:

$$\begin{pmatrix}
1 & 0 & & & \\
0 & 1 &  & & \\
 & (0) & \ddots & (0) & \\
 & & & N^2 & 0 \\
 &  & & 0 & N^2 \end{pmatrix}$$
hence the result.

\smallskip

Let us show now that $\varphi_N $ is a Cauchy sequence. Let $N\geq M \geq 0$ be two integers, we have : 

$$||\varphi_N - \varphi_M ||_{H^\sigma_R}^2 = \sum_{n=M}^N |g_n|^2 \frac{2n^{2\sigma}}{n^2}$$

$$||\varphi_N -\varphi_M||^2_{L^2_\omega,H^\sigma_R} = 2\sum_{n=N}^M \frac{1}{n^{2(1-\sigma)}}$$
as $\sigma<1/2$ implies $2(1-\sigma)> 1$, the series of general term $n^{-2(1-\sigma)}$ converges and we get the result.

\smallskip

\end{proof}

\begin{proposition}\label{finie} Almost surely, $\int_{0}^\pi \Omega^{\alpha-2} |\Re \varphi(\omega)|^{\alpha+2}\sin^2 R dR$ is finite. \end{proposition}

\begin{proof} Indeed, since $\Omega \leq 2$, to show that the probability of the event \\ ``$\int \Omega^{\alpha-2} |\Re \varphi(\omega)|^{\alpha+2}\sin^2 R dR=\infty$" is zero, we only have to prove that $E(||\varphi||_{L^{\alpha+2}_R}^{\alpha+2})$ is finite.

\smallskip

Since $E(||\varphi||_{L^{\alpha+2}_R}^{\alpha+2})=\int ||\varphi ||_{L^{\alpha+2}_\omega}^{\alpha+2}\sin^2 R dR $ and, for all $q$, thanks to lemma \reff{poumpoumpoum}, we have

$$||\sum c_n g_n ||_{L^q_\omega } \leq C\sqrt q (\sum |c_n|^2)^{1/2}$$

the norm $L^{\alpha+2}$ of $\varphi$ is less than $C (\sum \frac{2}{n^2} |e_n(x)|^2)^{1/2}$. We deduce from that :

$$ E(||\varphi||_{L^{\alpha+2}_R}^{\alpha+2}) \leq C ||\sum \frac{2}{n^2} |e_n|^2 ||_{L^{(\alpha+2)/2}}^{(\alpha+2)/2} \leq C \left(\sum \frac{2}{n^2} ||e_n||_{L^{\alpha+2}}^2\right)^{(\alpha+2)/2}$$

But we have seen that $||e_n||_{L^{\alpha+2}} = O(n^{1-\frac{3}{\alpha+2}})=O(n^{2/5})$ so the series converges. We get that $E(||\varphi||_{L^{\alpha+2}_R}^{\alpha+2})< \infty$, hence the result. 

\smallskip

\end{proof}

\smallskip

We can now define a new measure $\rho$ on $H^\sigma$.

\begin{definition}
We define on $H^\sigma$ with $\sigma < \frac{1}{2}$ the measure $\rho$ by : 

$$d\rho(u) = e^{-\frac{1}{\alpha+2}\int \Omega^{\alpha-2}|\Re u|^{\alpha+2}}d\mu(u)$$

\end{definition}

This measure is the limit of a sequence of measure on $E_N$ with the following meaning : 

\begin{definition}
First, we denote $\rho_N$ the measure on $E_N$ such that :

$$d\rho_N (u) = \exp \left(-\frac{1}{\alpha+2}\int \Omega^{\alpha-2}|S_N \Re u|^{\alpha+2}\right) d\mu_N (u)$$

\end{definition}

\begin{proposition} The application $u\mapsto e^{-\frac{1}{\alpha+2}\int \Omega^{\alpha-2}|S_N \Re u|^{\alpha+2}}$ converges in norm $L^1_{d\mu}$ towards \\ $e^{-\frac{1}{\alpha+2}\int \Omega^{\alpha-2}|\Re u|^{\alpha+2}}$. We deduce from that : $\lim \rho_N(E_N) = \rho(H^\sigma)$.
\end{proposition}

\begin{proof}Convergence of the $f_N(u) = \exp (-\frac{1}{\alpha+2}\int \Omega^{\alpha-2}|S_N \Re u|^{\alpha+2})$.

\smallskip

The proof uses the following lemma : 

\begin{lemma}Let $p \geq 2$, $\sigma<1/2$. And let $s<1/2$ if $p\leq 3$, $s<\frac{3}{p}-\frac{1}{2}$ otherwise. There exists $\beta(s), \lambda_0(p) >0$ such that for all $N\geq N_0 \geq 0$, we have : 

$$\mu(\lbrace u\in \Hs^\sigma \;|\; ||S_N u -S_{N_0}u||_{\W^{s,p}}> \lambda \rbrace) \leq \left \lbrace{\begin{tabular}{ll}
$ \exp(-c N_0^{\beta(s)}\lambda^2)$ & \mbox{if} $\lambda > \lambda_0(p)$ \\
$C \lambda^{-p} N_0^{-p\beta(s)/2}$ & \mbox{otherwise} \end{tabular} } \right. $$

\end{lemma}

We have

$$\int d\mu |f(u)-f_N(u)| = \int d\lambda \mu(|f(u)-f_N(u)| > \lambda) $$

But,

$$|f(u)-f_N(u)| \leq C |\; \left(\int \Omega^{\alpha-2}|\Re u|^{\alpha+2}\right)^{1/(\alpha+2)}-\left(\int \Omega^{\alpha-2}|\Re S_N u|^{\alpha+2}\right)^{1/(\alpha+2)}\; | \leq C ||u-S_N u||_{L^{\alpha+2}}$$
so

$$\int |f(u)-f_N(u)| d\mu \leq C \int \mu(||u-S_N u||_{L^{\alpha+2}}>\lambda)d\lambda $$

We apply the lemma for $p=\alpha+2,s=0$ (which is possible since $\alpha < 3$). As $S_N u$ converges towards $u$ in $L^1_\omega,L^{\alpha+2}_R$-norm , (see the proof of proposition \reff{finie} and replace $\varphi$ by $S_N \varphi$), we can do $N\rightarrow \infty$ in the lemma and replace the notation $N_0$ by the notation $N$. Then, we divide the integral into three parts : between $0$ and  $N^{-\gamma}$ with $0<\gamma< \Frac{\beta (\alpha+2)}{2(\alpha+1)}$, then between $N^{-\gamma}$ and $\lambda_0$, and at last $\lambda_0$ and $\infty$.

$$\int_{0}^{N^{-\gamma}} \mu(||u-S_N u||_{L^{\alpha+2}}>\lambda) d\lambda \leq \frac{1}{N^\gamma}$$

This integral converges towards $0$ when $N$ tends to $\infty$.

$$\int_{N^{-\gamma}}^{\lambda_0} \mu(||u-S_N u||_{L^{\alpha+2}}>\lambda) d\lambda \leq C \int_{N{-^\gamma}}^{\lambda_0} \frac{N^{-(\alpha+2)\beta /2}}{\lambda^{\alpha+2}} d\lambda \leq \frac{C}{\alpha+1}N^{-(\alpha+2)\beta / 2 +\gamma (\alpha+1)}$$
this integral converges towards $0$ with the choice we have made for $\gamma$.

$$\int_{\lambda_0}^\infty \mu(||u-S_N u||_{L^{\alpha+2}}>\lambda) d\lambda \leq   \int e^{-cN^\beta \lambda^2} d\lambda = \frac{1}{N^{\beta/2}}\int e^{-c\lambda^2}d\lambda $$
which converges towards $0$ when $N\rightarrow \infty$.

\smallskip

End of the proof of the $L^1$ convergence of $f_N$.

\bigskip

Let us show that $\rho_N(E_N)$ tends to $\rho(H^\sigma)$.

$$\rho_N (E_N) = \int f_N(\varphi_N(\omega)) d\omega = \int f_N(\varphi(\omega))d\omega \rightarrow \int f(\varphi(\omega))d\omega  = \rho(H^\sigma)$$

\end{proof}

\smallskip

Let us prove the lemma.

\bigskip

\begin{lemma} Let $\sigma < \Frac{1}{2}$, $p\geq 2$ and $0\leq s < \Frac{1}{2}$ if $p\leq 3$, $s< \Frac{1}{2}-\Frac{3}{p}$ otherwise. There exists $\beta(s) > 0$  and $\lambda_0 \geq 0$ and two constants $c,C > 0$ depending on $p$ such that for all couple of integers $N\geq N_0 \geq 1$,

\begin{equation} \mu (u\in H^{\sigma}\; |\; ||S_N u -S_{N_0} u||_{W^{s,p}} > \lambda) \leq  \left \lbrace{\begin{tabular}{ll} 
$ \Exp{-cN_0^{\beta(s)}\lambda^2}$ & \mbox{if} $\lambda\geq \lambda_0$ \\
$C\lambda^{-p} N_0^{-p\beta/2}$ & \mbox{otherwise}\end{tabular}} \right. \end{equation}

In particular, for $N_0 =1$, we get the property : there exists $C,c$ such that for all $\lambda \geq 1$ and $N\geq 1$

$$\mu (u\in \Hs^{\sigma}\; |\; ||S_N u||_{\W^{s,p}} > \lambda) \leq C e^{-c\lambda^2} .$$

\end{lemma}

\begin{proof}

$$\mu (u\in H^{\sigma}\; |\; ||S_N u -S_{N_0} u||_{W^{s,p}} > \lambda) = P(||\sum_n (\chi_S(\frac{n^2}{N^2})-\chi_S(\frac{n^2}{N_0^2}))\frac{\sqrt 2}{n}g_n(\omega)e_n(x)||_{W^{s,p}_x}>\lambda)$$

$$=P(||\sum_n (\chi_S(\frac{n^2}{N^2})-\chi_S(\frac{n^2}{N_0^2}))\frac{\sqrt 2}{n^{1-s}}g_n(\omega)e_n(x)||_{L^{p}_x}>\lambda)$$

Set $f(\omega,x) = \sum_n (\chi_S(\Frac{n^2}{N^2})-\chi_S(\Frac{n^2}{N_0^2}))\Frac{\sqrt 2}{n^{1-s}}g_n(\omega)e_n(x)$.

\smallskip

Let $q\geq p$. Thanks to a convexity inequality,

$$||f||_{L^q_\omega,L^p_x}\leq ||f||_{L^p_x,L^q_\omega} .$$

We have seen that

$$||\sum_n c_n g_n||_{L^q_\omega}\leq C_1 \sqrt q (\sum |c_n|^2)^{1/2}$$
so

$$||f||_{L^q_\omega}\leq C_1 \sqrt q (\sum_n |\chi_S(\frac{n^2}{N^2})-\chi_S(\frac{n^2}{N_0^2})|^2 \frac{2}{n^{2(1-s)}}|e_n(x)|^2)^{1/2}$$

With a triangle inequality,

$$||f||_{L^p_x,L^q_\omega}\leq C_1 \sqrt q (\sum_n |\chi_S(\frac{n^2}{N^2})-\chi_S(\frac{n^2}{N_0^2})|^2 \frac{2}{n^{2(1-s)}} ||\;|e_n|^2||_{L^{p/2}})^{1/2}$$

As $||\;|e_n|^2||_{L^{p/2}} = ||e_n||_{L^p}^2 $ is less than a constant independant from $n$ for $p<3$, than $C_p \log^2 n$ if $p=3$ and by $C_p n^{2-6/p}$ otherwise, we have

$$\sum_n |\chi_S(\frac{n^2}{N^2})-\chi_S(\frac{n^2}{N_0^2})|^2  \frac{2}{n^{2(1-s)}} ||\;|e_n|^2||_{L^{p/2}}\leq 
\left \lbrace{
\begin{tabular}{lll}
$C_p \sum_n |\chi_S(\frac{n^2}{N^2})-\chi_S(\frac{n^2}{N_0^2})|^2 \frac{2}{n^{2(1-s)}}$ & \mbox{if } $p<3$ \\
$C_p \sum_n |\chi_S(\frac{n^2}{N^2})-\chi_S(\frac{n^2}{N_0^2})|^2 \frac{2\log^2 n}{n^{2(1-s)}}$ & \mbox{if } $p=3$ \\
$C_p \sum_n |\chi_S(\frac{n^2}{N^2})-\chi_S(\frac{n^2}{N_0^2})|^2 \frac{2}{n^{6/p-2s}}$ & \mbox{otherwise} 
\end{tabular} }
\right. $$
The real number $s$ being stricly less than $1/2$ for $p\leq 3$ and than $\frac{3}{p}-\frac{1}{2}$ otherwise, there exists $\beta(s)>0$ such that for all $N\geq N_0 \geq 1$,

$$\sum_n |\chi_S(\frac{n^2}{N^2})-\chi_S(\frac{n^2}{N_0^2})|^2  \frac{2}{n^{2(1-s)}} ||\;|e_n|^2||_{L^{p/2}}\leq  C_3 N_0^{-\beta(s)}.$$

Finally,

$$||f||_{L^q_\omega,L^p_x}\leq C_4 \sqrt q  N_0^{-\beta(s)/2}.$$

We get then 

$$P(||f||_{L^p_x}> \lambda) = P(||f||_{L^p_x}^q>\lambda^q)\leq \lambda^{-q}||f||_{L^q_\omega,L_x^p}^q$$

$$P(||f||_{L^p_x}> \lambda) \leq \left(\frac{C_4(p)}{\lambda N_0^{\beta(s)/2}} \sqrt q\right)^q$$

For all $\lambda \geq \lambda_0 (p):= 2p/C_4(p)$, the real number $q= \Frac{\lambda^2 N_0^{\beta(s)}}{4 C_4^2}$ is more than $p$, which means that for $\lambda \geq \lambda_0 := 2p/C_4$ and $N\geq N_0 \geq 1$, by choosing $q=\Frac{\lambda^2 N_0^{\beta(s)}}{4C_4^2}$,

$$P(||f||_{L^p_x}> \lambda) \leq e^{-cN_0^{\beta (s)}\lambda^2}$$
with $c=\Frac{log 2}{4 C_4^2}$.

\smallskip

For $\lambda < \lambda_0 $ we choose $p=q$.

\smallskip

In the end, there exists $C$ such that for all $N\geq N_0 \geq 1$,

$$\mu (u\in H^{\sigma}\; |\; ||S_N u -S_{N_0} u||_{W^{s,p}} > \lambda) \leq 
\left \lbrace{
\begin{tabular}{ll}
$ \Exp{-cN_0^{\beta(s)}\lambda^2}$ & \mbox{if } $\lambda \geq \lambda_0$ \\
$ C\lambda^{-p} N_0^{-\beta(s)p/2}$ & \mbox{otherwise}
\end{tabular} }
\right. $$

\end{proof}

\bigskip

\begin{proposition}\label{decrease} For every set $A\subseteq E_N$, the image under the flow of $A$, ie $\Psi_N(0,t) A$ for any time $t\in [-\pi,\pi]$ satisfies : 

$$\mu_N(\Psi_N(0,t) A) \geq \rho_N (A)$$
\end{proposition}

\begin{proof} The Lebesgue measure on $E_N$ is well-defined (since $E_N$ has a finite dimension) and noted $du$.

$$\mu_N(\Psi_N(0,t) A) = \int d_N du 1_{\Psi_N(0,t)A}(u) e^{-\frac{1}{2}||Hu||_{L^2}^2}$$

Thanks to a computation similar to the one made on section 2.2 about $\E$, we get that $\E_N$ is decreasing under the flow so, $\E_N(0,u) \geq \E_N(t,\Psi(0,t)u)  \geq \frac{1}{2}||H\Psi_N(0,t)u||_{L^2}^2$. What is more, the Lebesgue measure is invariant under the flow so : 

$$\mu_N(\Psi_N(0,t) A) = \int d_N du 1_A (u) e^{-\frac{1}{2}||H\Psi_N(0,t)u||_{L^2}^2}$$

We get

$$\mu_N(\Psi_N(0,t) A) \geq \int d_N du 1_A(u) e^{-\E_N(0,u)} = \rho_N(A) .$$

\end{proof}
 
\subsection{Well-posedness for all $T$}
 
Now that we have a sequence of measures $\rho_N$ and $\rho$ with nice properties regarding the pseudo-Hamiltonian flows $\Psi_N(t_0,t)$, we will show the existence of a $\rho$-full measured set $\Sigma$, such that the flow $\Psi(0,t)$ is defined for all $t\in [-\pi,\pi]$ on this set. 

\begin{proposition} Let $p\in]2\alpha,6[$ and $s$ defined accordingly by $s=\Frac{3}{2}-\Frac{4}{p}$. Let $\sigma < \Frac{1}{2}$. For all integers $i$ and $N$, there exists a set $\Sigma_N^i$, $\rho_N$-measurable such that $\rho_N(E_N \smallsetminus \Sigma_N^i) = O(2^{-i})$ and $C\geq 0$ (independant from $i$ and from $N$) such that for all $u_0\in \Sigma_N^i$ and all $t\in]-\pi,\pi[$,

$$||S(t')\Psi_N(0,t)u_0||_{L^p(t',x\in [-\pi,\pi]\times S^3)} + ||\Psi_N(0,t)u_0||_{H^\sigma} \leq C \sqrt i$$

\end{proposition}

\begin{proof}

\smallskip

Let $D$ be a positive real number and $B_N^i(D) = \lbrace u \in E_N \; |\; ||S(t)u||_{L^p_{t,x}} + ||u||_{H^\sigma}\leq D\sqrt i\rbrace$.

\smallskip

Let's study the measurement $\mu_N(E_N\smallsetminus B_N^i(D))$. 

$$E_N\smallsetminus B_N^i(D) = \lbrace u \; | \; ||S(t)u||_{L^p_{t,x}} + ||u||_{H^\sigma} > D\sqrt i\rbrace $$

$$\subseteq \lbrace u \in E_N \;|\; ||S(t)u||_{L^p} > \frac{1}{2}D\sqrt i \rbrace \cup \lbrace u \in E_N \;|\; ||u||_{H^\sigma} > \frac{1}{2}D\sqrt i \rbrace$$
so $ \mu_N (E_N\smallsetminus B_N^i(D)) \leq \mu_N (\lbrace u \in E_N \;|\; ||S(t)u||_{L^p} > \frac{1}{2}D\sqrt i \rbrace) + \mu_N (\lbrace u \in E_N \;|\; ||u||_{H^\sigma} > \frac{1}{2}D\sqrt i \rbrace)$. But we have :

$$\mu_N (\lbrace u \in E_N \;|\; ||S(t)u||_{L^p} > \frac{1}{2}D\sqrt i \rbrace)\leq P( || \sum_{n=1}^N n^{- 1}g_n(\omega) e^{-int}e_n||_{L^p}>D\sqrt i/2)$$

$$ \leq P(|| \sum_{n=1}^N n^{- 1}g_n(\omega)e^{-int} e_n||_{L^p}^q>(D\sqrt i/2)^q)$$

$$\leq 2^q D^{-q} i^{-q/2}||\sum \frac{g_n}{n} e^{int}e_n||_{L^q_\omega,L^p_{x,t}}^q\; .$$

For all $q\geq p$, this implies that : 

$$\mu_N (\lbrace u \in E_N \;|\; ||S(t)u||_{L^p} > \frac{1}{2}D\sqrt i \rbrace)\leq 2^qD^{-q} i^{-q/2}||\sum \frac{g_n}{n} e^{int}e_n||_{L^p_{x,t},L^q_\omega}^q$$

$$\leq 2^q D^{-q} i^{-q/2}||\sqrt{q \sum \frac{|e_n|^2}{n^2}}||_{L^p}^q$$

$$\leq 2^q D^{-q} i^{-q/2}q^{q/2} (\sum ||e_n||_{L^p}^2 n^{-2})^{q/2}$$

As $p<6$, there exists $\nu(p)$ and $C(p)$ such that $||e_n||_{L^p}^2<C(p)n^{1-\nu(p)}$ so the sum converges, we get a triangle inequality that can be written

$$\mu_N (\lbrace u \in E_N \;|\; ||S(t)u||_{L^p} > \frac{1}{2}D\sqrt i \rbrace)\leq (\frac{C\sqrt q}{D\sqrt i})^q\; .$$

With $q=\Frac{D^2 i}{C^2 e^2}\geq p$, ie $D^2  \geq \Frac{C^2 e^2 p}{i}$, we get : 

$$\mu_N (\lbrace u \in E_N \;|\; ||S(t)u||_{L^p} > \frac{1}{2}D\sqrt i \rbrace) \leq e^{-\frac{D^2 i}{C^2 e^2}}$$

There exits $C$, such that for all $D\geq C e \sqrt p$ and all $i$, 

$$\mu_N (\lbrace u \in E_N \;|\; ||S(t)u||_{L^p} > \frac{1}{2}D\sqrt i \rbrace)\leq C e^{-cD^2i}$$

The same argument can be used with the norm $H^\sigma$ since the general term of the series becomes $n^{2(\sigma - 1)}$ and that $\sigma < \frac{1}{2}$, so we have : 

$$\rho_N (E_N\smallsetminus B_N^i(D)) \leq C e^{-cD^2 i}$$
with $C$ and $c$ independant from $D$, $i$ and $N$.

Let $\tau $ be the time defined in the local well posedness lemma, we have, for all $t_0,t$ such that $t-t_0\in [-\tau,\tau]$:

$$\Psi_N(t_0,t)(B_N^i(D)) \subseteq \lbrace u\in E_N \; |\; ||S(t')u||_{L^p_{t,x}}+||u||_{H^\sigma} \leq CD\sqrt i\rbrace$$

We set

$$\Sigma_N^i(D) = \bigcap_{k=-[\pi/\tau]}^{[\pi/\tau]} \Psi_N(0,k\tau)^{-1} (B_N^i(D))$$

The measurement of its complementary in $E_N$ satisfies : 

$$\rho_N(E_N\smallsetminus \Sigma_N^i(D)) \leq \sum_k \rho_N(E_N\smallsetminus \Psi_N(0,k\tau)^{-1} B_N^i(D))$$

Since $\rho_N(E_N\smallsetminus \Psi_N(0,k\tau)^{-1} B_N^i(D)) = \rho_N(\Psi_N(0,k\tau)^{-1}(E_N\smallsetminus B_N^i(D)) \leq \mu_N(E_N\smallsetminus B_N^i(D))$ by proposition \reff{decrease}, we get

$$\rho_N(E_N\smallsetminus \Sigma_N^i(D)) \leq (2[\frac{\pi}{\tau}]+1)C e^{-cD^2i}$$

What is more, $\tau $ is equal to $c(1+D\sqrt i)^{-\delta}$ so 

$$\rho_N(E_N\smallsetminus \Sigma_N^i(D)) \leq C (D\sqrt i)^{\delta} e^{-cD^2 i}\; .$$

We choose $D$ large enough ($ D^2 > \log 2 /c$) to have $CD^\delta i^{\delta/2}e^{-cD^2 i}\leq C' 2^{-i}$, $C'$ independa;nt from $N$ and $i$, and we note $\Sigma^i_N = \Sigma^i_N(D)$. We have $\rho_N(E_N\smallsetminus \Sigma_N^i) \leq C' 2^{-i}$.

\smallskip

One can see that for all $t$ in $[-\pi,\pi]$, $t$ can be written $k\tau + t_1$ with $t_1\in [-\tau,\tau]$, and $k\in \lbrace -[\pi,\tau],\hdots,[\pi/\tau] \rbrace$. So we get, for $u\in \Sigma_N^i$, $\Psi_N(0,t) u = \Psi_N(k\tau, k\tau+t_1)(\Psi_N(0,k\tau) u)$ and since $\Psi_N(0,k\tau) u\in B_N^i(D)$, $||S(t')\Psi_N(0,t) u||_{L^p_{t',x}}+ ||\Psi_N(0,t) u||_{H^\sigma}\leq CD\sqrt i$.

\smallskip

\end{proof}

\begin{definition}Let $\widetilde \Sigma_i^N = \lbrace u \in H^{\sigma} \; |\; \Pi_N u\in \Sigma_N^i\rbrace $ and $\Sigma^i = \limsup \widetilde \Sigma^i_N$. For all $u\in \Sigma^i$ there exists a sequence $u_k \in \Sigma_{N_k}^i$ ($N_k\rightarrow \infty$) such that $u_k$ converges toward $u$ in $H^\sigma$.

\smallskip

We also write $\Sigma = \bigcup_i \Sigma^i$.

\end{definition}

\begin{proposition} The set $\Sigma $ is of full measure.\end{proposition}

\begin{proof} We have $\rho(\Sigma_i) \geq \limsup \rho(\widetilde \Sigma_N^i)$. 

$$\rho_N(\Sigma^i_N)= \int_{\Sigma^i_N}f_N(u) d\mu_N$$

$$= \int  \mu_N\left(\Sigma^i_N \cap f_N^{-1}(]\lambda,\infty])\right)d\lambda$$

$$= \int \mu\left(\Pi_N^{-1}(\Sigma^i_N \cap f_N^{-1}(]\lambda,\infty]))\right)d\lambda$$

$$=\int \mu\left(\widetilde\Sigma^i_N \cap f_N^{-1}(]\lambda,\infty])\right)d\lambda$$

$$= \int_{\widetilde \Sigma_N^i} f_N(u)d\mu(u)\; .$$

Remind that $\rho(\widetilde \Sigma_N^i) = \int_{\widetilde \Sigma_N^i} f(u)d\mu(u)$.

\smallskip

Since $f_N$ converges in $L^1_\mu$ toward $f$, $\limsup \rho(\widetilde \Sigma^i_N) = \limsup \int_{\widetilde \Sigma^i_N} f_N(u)d\mu = \limsup \rho_N(\Sigma_N^i)$. What is more, $\limsup \rho_N(\Sigma^i_N)\geq \lim \rho_N(E_N)-C2^{-i} = \rho(H^\sigma)- C2^{-i}$. We deduce that $\rho(\Sigma) \geq \rho(H^\sigma) - C2^{-i}$ for all $i$, and so $\rho(\Sigma)=\rho(H^\sigma)$.
\end{proof}

\begin{proposition} Let $p\in ]2\alpha,6[$ and $s$ defined accordingly as $s=\Frac{3}{2}-\Frac{4}{p}$. Let also $\sigma < \frac{1}{2}$. For all $u_0 \in \Sigma$, there exists a strong solution for all $T\in [-\pi,\pi]$ of the hamiltonian problem.\end{proposition}

\begin{proof} Let us begin with setting $i$ the integer such that $u_0\in \Sigma_i$, there exists a sequence $N_k \rightarrow \infty$ such that $u_{0,k} := \Pi_{N_k}u_0 \in \Sigma^i_N$. As the norms $||u_{0,k}||_{H^\sigma}$ are bounded by $D\sqrt i$, and that the sequence $u_{0,k}$ converges toward $u_0$ in norm $H^\sigma$, we have $||u_0||_{H^\sigma}\leq D \sqrt i$. 

\smallskip

We also have that the norms $||S(t)u_{0,k}||_{L^p}$ are bounded by $D\sqrt i$, and that $S(t)u_{0,k}$ converges as distributions toward  $S(t)u_0$ so $S(t)u_{0,k}$ converges toward $S(t)u_0$ in norm $L^p$. 

\smallskip

Set $u_k(t) = \Psi_{N_k}(0,t)u_{0,k}$. For all $t\in [-\pi,\pi]$ and all $k$ we have

$$||S(t')u_k(t)||_{L^p_{t',x}}+||u_k(t)||_{H^\sigma}\leq C' \sqrt i.$$
what is more,

$$||S(t')u_0||_{L^p}+||u_0||_{H^\sigma}\leq C'\sqrt i$$

\begin{lemma}Let $t_0\in [0,\pi]$ and suppose that $\Psi(0,t)u_0$ is defined for all $t\in[0,t_0]$, and satisfies

$$||S(t')\Psi(0,t_0)u_0||_{L^p}+||\Psi(0,t_0)u_0||_{H^\sigma}\leq C'\sqrt i \; .$$ 

We assume that $\Psi(0,t)u_0$ is the unique solution of \reff{hamil} in the sense that $\Psi(0,t)u_0 - S(t)u_0$ is unique in $X^s_{t_0}$.

\smallskip

We suppose also that $u_k(t_0)$ converges toward $\Psi(0,t_0)u_0$ in norms $||\; .\; ||_{H^\sigma_R}$ and $||S(t) .\;||_{L^p_{t,R}}$.

\smallskip

Then, there exists a time $\tau'$ independant from $t_0$ such that the solution $\Psi(0,t)u_0$ is defined and unique in the sense that $\Psi(0,t)u_0-S(t)u_0$ is unique in $X^s_{t_0+\tau'}$ for all $t\in[0,t_0+\tau']$ and satisfies :

$$||S(t')\Psi(0,t_0+\tau')u_0||_{L^p_{t',R}}+||\Psi(0,t_0+\tau')u_0||_{H^\sigma_R}\leq C'\sqrt i \; .$$ 

We get the same convergences at time $t_0+\tau'$.
\end{lemma}

\begin{proof}

Let $A=C'\sqrt i$ and $\tau $ the corresponding time involved in the local well-posedness lemma. On $[t_0-\tau,t_0+\tau]$, we can define $u$ the local solution and $v(t)=u(t+t_0)-S(t)\Psi(0,t_0)u_0$. As $v$ is unique from the local well-posedness lemma in $X^s_\tau$, we have that

$$v(t) = \Psi(0,t+t_0)u_0-S(t+t_0)u_0 + S(t)\left(\Psi(0,t_0)u_0 - S(t_0)u_0\right)\; ,$$
and $S(t)\left(\Psi(0,t_0)u_0 - S(t_0)u_0\right)$ is in $X^s_{\tau+t_0}$ by hypothesis, we have that $\Psi(0,t)u_0-S(t)u_0$ is unique in $X^s_{t_0+\tau}$.

\smallskip

We write the analoguous functions $v_k(t) = u_k(t+t_0)-S(t)u_{k}(t_0)$, and $w_k = v-v_k$. Remark that from the local well-posedness property, for all $t\in[-\tau,\tau]$ : 

$$||v||_{X^s_\tau} + ||u(t+t_0)||_{H^\sigma}+||S(t')u(t+t_0)||_{L^p_{t',x}} \leq C A,$$

$$||v_k||_{X^s_\tau} + ||u_k(t_0+t)||_{H^\sigma}+||S(t')u_k(t_0+t)||_{L^p_{t',x}} \leq C A.$$

$w_k$ can be written as : 

$$w_k = v-v_k = -i\int_{0}^t dt'S(t-t')H^{-1}(F(t'+t_0,u)-S_{N_k} F(t'+t_0,S_{N_k} u_k))$$

$$w_k = -i\int_{0}^t dt' S(t-t')H^{-1}S_{N_k}(F(t'+t_0,u)-F(t'+t_0,S_{N_k}u)) -i\int_{0}^t dt' S(t-t')H^{-1}(1-S_{N_k})F(t'+t_0,u)$$

For $\tau'\leq \tau$, the norm $X^s_{\tau'}$ of the first integral is less than : 

$$||-i\int_{0}^t dt' S(t-t')H^{-1}S_{N_k}(F(t'+t_0,u)-F(t'+t_0,S_{N_k}u))||_{X^s_{\tau'}}$$

$$\leq C(\tau')^\delta A^2(||S(t)(\Psi(0,t_0)u_0-S_{N_k}u_{k}(t_0)||_{L^p} + ||v-S_{N_k}v_k||_{X^s_{\tau'}})$$

$$\leq C(\tau')^\delta A^2||w_k||_{X^s_{\tau'}}$$
with $\delta > 0$ chosen as precedently.

\smallskip

Regarding the second integral, we use the fact that for $s<s_1<(\alpha +1)s-\Frac{3\alpha-4}{2}$ (which is possible since $p>2\alpha$), 

$$||(1-S_{N_k})\int S(t'-t)H^{-1} F(t'+t_0,u)dt' ||_{X^s_{\tau'}}\leq CN_k^{s-s_1}||F(t'+t_0,u)||_{Y_{\tau'}^{1-s_1}}$$

By choosing $p_1$ such that $\Frac{4}{p_1}=\Frac{3}{2}-1+s_1=\Frac{1}{2}+s_1$, that is to say $\Frac{4}{p'_1}=\Frac{7}{2}-s_1$, we have $||u||_{L^{(\alpha+1)p'_1}}^{\alpha+1}\geq C||F(u)||_{L^{p'_1}}\geq C||F(u)||_{Y_{\tau'}^{1-s_1}}$. But, 

$$\frac{4}{(\alpha+1)p'_1}=\frac{7}{2(\alpha+1)}-\frac{s_1}{\alpha+1}$$

$$\frac{4}{(\alpha+1)p'_1}\geq \frac{7}{2(1+\alpha)}-s+\frac{3\alpha-4}{2(1+\alpha)}= \frac{3}{2}-s=\frac{4}{p}$$

So

$$||(1-S_{N_k})\int S(t'-t)H^{-1} F(t'+t_0,u)dt' ||_{X^s_{\tau'}}\leq CN_k^{s-s_1}(\tau')^{\delta_1} (||\Psi(0,t_0)u_0||_{H^\sigma}+||v||_{X^s_{\tau'}})$$
with $\delta_1 = \Frac{1}{(\alpha+1)p'_1}-\Frac{1}{p}$. 

\smallskip

Hence,

$$||w_k||_{X^s_{\tau'}}\leq C(\tau')^{\delta}A^2||w_k||_{X^s_{\tau'}}+ C^2N_k^{s-s_1}(\tau')^{\delta_1}A \; .$$

By choosing $\tau'$ small enough such that $C(\tau')^{\delta_1} A^2 < 1$, we get that the norm of $w_k$ converges toward $0$ when $k\rightarrow \infty$. We deduce that for all $t\in[-\tau',\tau']$, 

$$||u(t_0+t)-u_k(t_0+t)||_{H^\sigma}\leq ||S(t)(u(t_0))-u_k(t_0))||_{H^\sigma} + ||u(t_0+t)-u_k(t_0+t)-S(t)(u(t_0)-u_k(t_0))||_{X^s_{\tau'}} $$

$$ = ||S(t)\left(u(t_0)-u_k(t_0)\right)||_{H^\sigma}+ ||w_k||_{X^s_{\tau'}}$$

$$= ||u(t_0)-u_k(t_0)||_{H^\sigma}+ ||w_k||_{X^s_{\tau'}}\rightarrow 0 \; .$$

What is more, $S(t')(u(t_0+\tau')-u_k(t_0+\tau'))\rightarrow 0$ in norm $L^p_{t',R}$. We get that 

$$||S(t')u(t_0+\tau')||_{L^p}+||u(t_0+\tau')||_{H^\sigma}\leq \lim ||S(t')u_k(t_0+\tau')||_{L^p_{t',R}}+||u_k(t_0+\tau')||_{H^\sigma}\leq A = C'\sqrt i \; .$$

\end{proof}

To finish the proof of the proposition, we can then apply the lemma with \\ $t_0 = 0,\tau',2\tau'\hdots,M\tau'$, with $M = [\pi/\tau']$ and use a similar argument for the negative times.
\end{proof}

\subsection{Back to the non linear wave equation on $\R^3$}

Let us see now what this result means for the non linear wave equation on the Euclidian space. To understand such a thing, we have to wonder to which spaces belongs the Penrose transform at time $T=0 \Leftrightarrow t=0$ of the initial data $u_0$ that is to say, an initial data taken in $L^2$.

\begin{definition} Let $PT$ the transform defined on $L^2$ by : 

$$u_0 \in L^2 \mapsto f_0,f_1$$
with 

$$f_0(r) = \frac{2}{1+r^2}\Re u_0(2\arct(r)) \; \mbox{and} \; f_1(r) = -\left(\frac{2}{1+r^2}\right)^2 (H\Im u_0)(2\arct(r))\; .$$
\end{definition}

\begin{remark} The transform $PT$ is the Penrose transform taken at time $t=0$. That is to say, if $f(t,r)$ is the Penrose transform of $u(T,R)$ then $f(t=0,r) = \frac{2}{1+r^2}\Re u_0(2\arct(r))$ and $\partial_t f (t=0,r) = -\left(\frac{2}{1+r^2}\right)^2 (H\Im u_0)(2\arct(r))$.\end{remark}

We define now the spaces which the $PT$ sends $L^2$ to. 

\begin{definition}Let $m\in \R$, we define $L^2_m$ (resp. $H^{-2}_m$) as the set of radial distributions $f$ on $\R^3$ such that the $L^2$ (resp. $H^{-2}$) -norm of $(\frac{1+r^2}{2})^{m/2}f$ is finite. This is a normed vector space and its norm is : 

$$ ||f||_{L^2_m} = ||\left( \frac{1+r^2}{2} \right)^{m/2} f ||_{L^2}$$
(resp.

$$||f||_{H^{-2}_m} = ||\left( \frac{1+r^2}{2}\right)^{m/2} f ||_{H^{-2}} \; )\; . $$
\end{definition}

\begin{proposition}The transform $PT$ continuously sends $L^2_{S^3}$, meaning the $L^2$ space of zonal complex functions of the sphere $S^3$ into $L^2_{-1}\times H^{-2}_{-6}$ wich means the $L^2$ space times the $-1$ Sobolev space of radial real functions of the Euclidian space $\R^3$ with weights.  \end{proposition}

\begin{proof}To sort this out, let us consider the following lemma : 

\begin{lemma} Let $u$ be in $L^2_{S^3}$ and $v$ the radial distribution on the Euclidian space defined by a change of varaiable as : 

$$v(r) = u(2\arct(r)) \; .$$

The variable change is an isometry between $L^2_{S^3}$ and $L^2_{-3}$. \end{lemma}

\begin{proof} Let us compute $||v||_{L^2_{-3}}$.

$$||v||_{L^2_{-3}}^2 = ||\left( \frac{1+r^2}{2}\right)^{-3/2} v||_{L^2}^2$$

$$||v||_{L^2_{-3}}^2 = \int_{0}^\infty |v(r)|^2 \left( \frac{2r}{1+r^2}\right)^2 \frac{2dr}{1+r^2}$$

By the change of variable $R= 2\arct(r)$, we get : 

$$||v||_{L^2_{-3}}^2 = \int_{0}^\pi |u(R)|^2 \sin^2 (R) dR = ||u||_{L^2_{S^3}}^2 \; .$$

\end{proof}

Let us first deal with the norm of $f_0$. This distribution is : 

$$f_0 = \frac{2}{1+r^2} \Re u_0 (2\arct(r)) \; .$$

And let $v_0(r) = \Re u_0 (2\arct(r))$. The $L^2_{-1}$ norm of $f_0$ is

$$||f_0||_{L^2_{-1}} = ||\left(\frac{1+r^2}{2} \right) ^{-1/2-1} v_0||_{L^2}= ||v_0||_{L^2_{-3}}=||\Re u_0||_{L^2} \; .$$

For $f_1 = (\frac{2}{1+r^2})^2 (H\Im u_0) (2\arct(r)) $, consider $\Delta_{S^3}$ the Laplace-Beltrami operator on the sphere $S^3$. When acting on zonal functions, this operator reduces into $\partial_R^2 + \frac{2}{\tan R}\partial_R$. Let us now do the variable change $R=2\arct(r)$. First, set $u$ a zonal function of $S^3$ and $v$ a radial one such that $v(r) = u(2\arct(r))$ and compute the action of $(1-\Delta_{S^3})$ after the variable change $R=2\arct(r)$.

$$\frac{du}{dR} = \frac{1+\tan^2(R/2)}{2}\frac{dv}{dr} = \frac{1+r^2}{2} \frac{dv}{dr} \; ,$$

$$\frac{d^2u}{dR^2} = \left(\frac{1+r^2}{2}\right)^2 \frac{d^2v}{dr^2}+\frac{1+r^2}{2}r \frac{dv}{dr} \; ,$$

$$\Delta_{S^3} u= \left(\frac{1+r^2}{2}\right)^2 \frac{d^2v}{dr^2}+\frac{1+r^2}{2r} \frac{dv}{dr} \; .$$

We call 

$$D^2 = 1 - \left(\frac{1+r^2}{2}\right)^2 \frac{d^2}{dr^2} - \frac{1+r^2}{2r}\frac{d}{dr} \; .$$

Let $\varphi$ be a radial function in $H^2(\R^3)$ and compute the distribution bracket : 

$$| \langle \left( \frac{1+r^2}{2} \right)^{-3} f_1 \; , \; \varphi \rangle_{\R^3} | \;   .$$

By definition, this is equal to : 

$$| \int_{0}^\infty \left( \frac{1+r^2}{2} \right)^{-3} f_1(r) \varphi(r) r^2 dr |$$
and by a variable change : 

$$|\int_{0}^\pi \left(\frac{1+r^2}{2}\right)^{-2} (H\Im u_0)(R) \psi(R) \sin^2(R) dR |$$
where $\Psi(R) = \varphi(\tan(\frac{R}{2}))$ and $r=\tan(\frac{R}{2})$.

\smallskip

As $H$ is a self-adjoint operator in $L^2_{S^3}$ we get that the bracket is equal to : 

$$| \int_{0}^{\pi}\Im u_0 (R) H \left( \left(\frac{1+r^2}{2} \right)^{-2} \psi \right) \sin^2 R dR |$$
which is less than : 

$$||u_0||_{L^2_{S^3}} || \left(\frac{1+r^2}{2} \right)^{-2} \psi ||_{H^1_{S^3}}$$
and also

$$|\langle \left(\frac{1+r^2}{2}\right)^{-3} f_1, \varphi\rangle_{\R^3} | \leq C ||u_0||_{L^2_{S^3}} || \left(\frac{1+r^2}{2} \right)^{-2} \psi ||_{H^2_{S^3}}$$
since $H^2_{S^3}$ is continuously embedded in $H^1_{S^3}$.

\smallskip

So now, we have to prove that $\left(\frac{1+r^2}{2} \right)^{-2} \psi$ is in $H^2_{S^3}$, that is to say that 

$$(1-\Delta_{S^3}) \left( \left(\frac{1+r^2}{2} \right)^{-2} \psi \right)$$
is in $L^2_{S^3}$.

\smallskip

First, by commuting $(1-\Delta_{S^3})$ and $\left(\frac{1+r^2}{2} \right)^{-2}$, we get that (remind that $r$ is considered asa function of $R$) : 

$$(1-\Delta_{S^3}) \left( \left(\frac{1+r^2}{2} \right)^{-2} \psi \right)(R) = \left(\frac{1+r^2}{2} \right)^{-2} (1-\Delta_{S^3})\psi(R)$$

$$+ (4-16r^2) \left(\frac{1+r^2}{2} \right)^{-2} \psi(R)$$

$$+ 8r \left(\frac{1+r^2}{2} \right)^{-2} \frac{d\psi}{dR}(R)\; .$$

Hence, since $r=\tan(\frac{R}{2})$ we get that the $L^2_{S^3}$-norm of $(1-\Delta_{S^3}) \left( \left(\frac{1+r^2}{2} \right)^{-2} \psi \right)(R)$ is less than the $L^2_{S^3}$ the norm of $C\left(\frac{1+r^2}{2}\right)^{-1/2} (1-\Delta_{S^3}) \psi$, $C$ independant from $\psi$ (the multiplying weights are smaller than a certain constant).

\smallskip

Also, we have seen that $(1-\Delta_{S^3})\psi (2\arct (r))$ was equal to $D^2 \varphi (r)$. Let us compute the $L^2_{S^3}$ norm of $\left(\frac{1+r^2}{2}\right)^{-1/2} (1-\Delta_{S^3}) \psi$.

$$||\left(\frac{1+r^2}{2}\right)^{-1/2} (1-\Delta_{S^3}) \psi||_{L^2_{S^3}}^2 = \int_{0}^\pi \left(\frac{1+r^2}{2}\right)^{-1} |(1-\Delta_{S^3})\psi|^2(R) \sin^2 R dR \; .$$

By a variable change $r=\tan(\frac{R}{2})$, we get : 

$$|| \left(\frac{1+r^2}{2}\right)^{-1/2} (1-\Delta_{S^3}) \psi||_{L^2_{S^3}}^2 = \int_{0}^\infty \left(\frac{1+r^2}{2}\right)^{-4} |(1-\Delta_{S^3})\psi|^2(2\arct(r)) r^2 dr$$

$$= ||\left( \frac{1+r^2}{2}\right)^{-2} D^2 \varphi ||_{L^2_{\R^3}}^2\; .$$

But as 

$$\left( \frac{1+r^2}{2}\right)^{-2} D^2 \varphi = \left(\frac{1+r^2}{2}\right)^{-2}\varphi -\Delta_{\R^3}\varphi + \frac{2r}{1+r^2}\frac{d\varphi}{dr}$$

we get that : 

$$||\left( \frac{1+r^2}{2}\right)^{-2} D^2 \varphi ||_{L^2_{\R^3}} \leq ||\varphi||_{H^2_{\R^3}} \;. $$

Let us go back to the distribution bracket. There exists $C$ independant from $u_0$ such that for all $\varphi \in H^2$, we have : 

$$| \langle \left( \frac{1+r^2}{2} \right)^{-3} f_1 \; , \; \varphi \rangle | \leq C ||\Im u_0||_{L^2} ||\varphi||_{H^2}$$
so $\left( \frac{1+r^2}{2} \right)^{-3} f_1$ is indeed in $H^{-2}$ with a norm less than $C ||u_0||_{L^2_{S^3}}$, that is to say $f_1$ is in $H^{-2}_{-6}$ and its norm satisfies : 

$$||f_1||_{H^{-2}_{-6}} \leq C ||u_0||_{L^2}$$
with a constant $C$ independant from $u_0$. 

Hence, the norm of $(f_0,f_1)$ in the cartesian product $L^2_{-1}\times H^{-2}_{-6}$ satisfies : 

$$||f_0,f_1||\leq C ||u_0||_{L^2_{S^3}}$$
with $C$ independant from $u_0$, so the transform $PT$ is continuous from $L^2_{S^3}$ to $L^2_{-1}\times H^{-1}_{-1}$. 

\end{proof}

Let us now define a measure $\eta$ for this set.

\begin{definition}Let $\mathcal H=L^2_{-1} \times H^{-2}_{-6}$. We call $\eta $ the image measure by $PT$ on $\mathcal H$ of $\rho$. 

\smallskip

We also set $\Pi$ the set included in $H$ such that :

$$\Pi = PT(\Sigma)$$
where $\Sigma$ is the previously defined set of full $\rho$ measure of $L^2_{S^3}$ onto which the flow is globally defined.
\end{definition}

\begin{theorem}\label{globsol} Let  $p\in ]2\alpha,6[$ and $s$ defined accordingly as $\Frac{3}{2}-\Frac{4}{p}$. The set $\Pi$ is of full $\eta$ measure and for all couple $f_0,f_1 \in \Pi$, we have a solution $f$ of \reff{nlwe} with initial data $f_0,f_1$. This solution is unique in the sense that $f(t,r)-L(t)(f_0,f_1)(r)$ is unique in $L^p_{t,r}$, where $L(t)$ is the flow of the linear wave equation, i.e. the flow of $\partial_t^2-\Delta_{\R^3}$.\end{theorem}

\begin{proof}First, $\Pi$ is of full $\eta$-measure since $\Sigma$ is of full $\rho$-measure. Then Let $u_0$ be $(PT)^{-1}(f_0,f_1)$, that is to say :

$$u_0 = \frac{1}{1+\cos R} f_0(\tan(\frac{R}{2})) -i H^{-1}\left( \left(\frac{1}{1+\cos R}\right)^2 f_1(\tan(\frac{R}{2})) \right) \; .$$

The function $u_0$ is therefore in $\Sigma$. Hence, there exists $u$ such that $u-S(t)u_0$ is unique in $X^s_\pi$ satisfying \reff{hamil}. We then choose $f$ the time dependant Penrose transform of $u$. The function $f$ satisfies \reff{nlwe} with initial condition $f_0,f_1$, as we have already discussed it in the previous sections. 

\smallskip

We also have that $L(t)(f_0,f_1)$ is the reverse Penrose transform of $S(T)u_0$.

\begin{lemma}\label{free} Let  $f_\infty = f_0,f_1 \in \mathcal H $ and $f$ the solution of

\begin{equation}\label{libre}\left \lbrace {\begin{tabular}{ll}
$(\partial_t^2 -\Delta) f = 0$ & \\
$f|_{t=0}=f_1$ & $\partial_t f|_{t=0}= f_1$ \end{tabular}}
\right. \end{equation}

We set $u_\infty (R) = \Frac{f_0(r)}{\Omega(0,R)} -i H^{-1}\Frac{f_1(r)}{\Omega^2(0,R)}\in L^2_{S^3}$ and $u = S(T)u_\infty$ the solution of : 

$$\left \lbrace{\begin{tabular}{ll}
$i\partial_T u + Hu = 0$\\
$u_{T=0}=u_\infty$ \end{tabular}}
\right. $$
we have $f(t,r)=\Omega \Re u(\arct(t+r) + \arct(t-r)\; , \; \arct(t+r) - \arct(t-r))$.
\end{lemma}

\begin{proof} Indeed, $v(T,R)=\Frac{f(t,r)}{\Omega}$ satisfies :

$$\left \lbrace{\begin{tabular}{ll}
$(\partial_T^2 +1 -\Delta) v = 0$ & \\
$v_{T=0}=\frac{f_0}{\Omega}$ & $\partial_T v|_{T=0} = \frac{f_1}{\Omega}$ \end{tabular}}
\right. $$

So, $\widetilde u = v-iH^{-1} \partial_T v$ satisfies

$$\left \lbrace{\begin{tabular}{ll}
$i\partial_T \widetilde u + H\widetilde u = 0$\\
$\widetilde u_{T=0}=u_\infty$ \end{tabular}}
\right. $$
which means $\widetilde u = u$. Since $f=\Omega v$ and $v=\Re \widetilde u$, we have $f=\Omega \Re u$.

\end{proof}

\smallskip

Let us now prove the uniqueness in $L^p_{t,r}$. Let $f$ and $g$ be two solutions of \reff{nlwe} and $u$ and $v$ their respective Penrose transform, that is to say : 

$$u(T,R) = \Omega^{-1}f(t,r) - i \partial_T H^{-1}\Omega^{-1}f (t,r)\; ,$$

$$v(T,R) = \Omega^{-1}g(t,r) - i \partial_T H^{-1}\Omega^{-1}g (t,r)\; .$$

Let us compute the $L^p_{t,r}$ norm of $f-g$.

\smallskip

First, let us consider a function $\phi$ in $L^p_{T,R}$ and set $\psi$ its Penrose transform. We have :

$$\int_{0}^\infty \int_{-\infty}^\infty |\psi(t,r)|^p r^2dr dt = \int \int_{\Omega>0} \frac{1}{\Omega^2} |\Omega \phi(T,R)|^p \frac{\sin^2 R}{\Omega^2} dR$$
since the Jacobian of the Penrose transform is equal to $\Frac{1}{\Omega^2}$ and $r=\Frac{\sin R}{\Omega}$. 

\smallskip

Hence, we get that : 

$$||\psi ||_{L^p_{t,r}} = ||\Omega^{1-4/p}\phi||_{L^p_{T,R}}\leq C ||\phi||_{L^p_{T,R}}$$
as $p>2\alpha \geq 4$.

\smallskip

The Penrose tranform is linear, hence we get that $f-L(t)(f_0,f_1)$ is the Penrose transform of $u-S(t)u_0$, which is in $L^p_{T,R}$. Therefore $f-L(t)(f_0,f_1)$ is in $L^p_{t,r}$. Let us compute $f-g = f-L(t)(f_0,f_1) - (g-L(t)(f_0,f_1))$, ie the Penrose transform of $u-S(t)u_0 - (v-S(t) u_0)$. We then have : 

$$||f-g||_{L^p_{t,r}} \leq C ||u-v||_{L^p_{T,R}} \leq C ||u-v||_{X^s_\pi}$$
by the embedding of $X^s_\pi$ into $L^p_{T,R}$. We deduce $||f-g||_{L^p}=0$, hence the unicity. \end{proof}
 
\section{Typical properties of the solutions}

\subsection{General considerations}

We have seen that the initial data is of the form $u_0 = \sum \Frac{\sqrt 2 g_n}{n} e_n$ with $g_n$ complex gaussian random variables. But remember that $u_0 = v_0 -i H^{-1} v_1$, where $v_0$, $v_1$ are the initial data for the problem outside its Hamiltonian form, that is to say : $v|_{T=0} = v_0$, $\partial_T v|_{T=0}= v_1$.

\smallskip

What is more, the initial data for the problem before using the Penrose tranform was given by $f_0 = \Frac{2}{1+r^2} v_0$ and $f_1 = (\Frac{2}{1+r^2})^2 v_1$.

\smallskip

The function $v_0 = \Re u_0$ can be written $\sum \Frac{h_n}{n} e_n$ where $h_n$ are independant real gaussian random variables, and $v_1 = -H\Im u_0$ is $\sum l_n e_n$ where $l_n$ are also independant real gaussian variables (and independant from the $h_n$). In other terms, $v_1$ is in the same spaces as $H v_0$. But so, $f_1 = \Frac{4}{(1+r^2)^2}v_1$ is in the same spaces as $\Frac{4}{(1+r^2)^2}H v_0=\Frac{4}{(1+r^2)^2}H \Frac{1+r^2}{2}f_0$ where $H$ is expressed in terms of $t,r$ instead of $T,R$.

\smallskip

That is to say, since $H^2 = 1-\Delta_{S^3} = 1-\partial_R^2 - \Frac{2}{\tan R}\partial_R$, that $f_1$ lives in the same spaces as

$$(\frac{2}{1+r^2})^2 \sqrt{1-\left(\frac{1+r^2}{2}\right)^2 \partial_r^2 - \frac{(1+r^2)}{2r}\partial_r} \frac{1+r^2}{2}f_0\; .$$

We are going to show that $f_0$ is almost surely in $L^p$, for $p\in ]2,6[$, and that it is almost surely not in $L^p$, when $p$ is different. Then, we will use techniques of fractional integration on periodic functions and their Fourier decomposition to both characterize the behavior of $f_0(r)$ when $r\rightarrow \infty$ and  show that $f_1$ belongs to the spaces $W^{-1,p}$ almost surely for all $p\in ]2,6[$.

%\smallskip

%Moreover, the divergence of the norm $L^p$ for $p\geq 6$ is due to the behaviour of $f_0$ for small $r$. If one considers $f_0$ outside a neighborhood of $0$, $f_0$ appears to be in $L^p$ for all $p>2$, and in particular in $L^\infty$. We'll see that $f_0(r)\rightarrow 0$ when $r\rightarrow \infty$.

\smallskip

To sum up, we are going to prove in the next subsections the following properties :

\begin{proposition} The initial data $f_0 = \sum \Frac{h_n}{n} \Frac{2}{1+r^2} e_n(2\arct(r))$ is almost surely in $L^p$ for $p\in ]2,6[$, $f_0|_{r>1} \in L^p$ almost surely for $p>2$, and $f_0(r)$ converges towards $0$ almost surely in $\omega$ when $r$ tends to $\infty$.
\end{proposition}

\begin{proposition}Almost surely, $f_0$ does not belong to $L^p$ for $p\leq 2$ and $p\geq 6$.\end{proposition}

\begin{proposition} The $t$ derivative at $t=0$ $f_1$ is almost surely in $W^{-1,p}$ for $p\in ]2,6[$.\end{proposition}

\subsection{$L^p$-Spaces the initial data belong to}

\begin{definition}
Let $f_n$ be the Penrose transform at time $t=0 \Leftrightarrow T=0$ of $e_n$. The functions $f_n$ can be written : 

$$f_n(r) = \Omega (t=0) e_n(2\arct r) = \sqrt{\frac{2}{\pi}} \frac{2}{1+r^2}\frac{\sin (2n\arct r)}{\sin(2\arct r)}$$

$$f_n(r) = \sqrt{\frac{2}{\pi}} \frac{\sin (2n\arct r)}{r}\; .$$

The initial data $f_0$ on $\R\times (\R^3)$ is then : 

$$f_0(\omega,r) = \sum_n h_n(\omega) \frac{f_n(r)}{n} ,$$
with $h_n$ independant real gaussian variables of law $\mathcal N (0,1)$.
\end{definition}

\begin{proposition} We have the following inequalities :
$$||f_n ||_{L^p} \leq \left \lbrace{\begin{tabular}{ll} 
$C_p n^{3/p - 1}$ & \mbox{if} $p< 3$ \\
$C_p \log n $ & \mbox{if} $p=3$\\
$C_p n^{1-3/p}$ & \mbox{otherwise} \end{tabular} } \right. $$
\end{proposition}

\begin{proof} By definition, 
$$||f_n||_{L^p}^p = C\int_{0}^\infty  |\frac{\sin (2n \arct r)}{r}|^p r^2 dr$$

By changing $r$ into $R=\tan r$, we get :

$$||f_n||_{L^p}^p = \int_{0}^{\pi/2} |\tan R|^{2-p} |\sin 2n R|^p (1+\tan^2 R) dR$$

This integral has two singularities : on $0$ ($r=0$) and $\pi/2$ ($r=\infty$), we are going to treat them separately by dividing the integral in two parts. We note

$$I = \int_{0}^{\pi/4} |\tan R|^{2-p} |\sin 2n R|^p (1+\tan^2 R) dR$$
and

$$II=\int_{\pi/4}^{\pi/2} |\tan R|^{2-p} |\sin 2n R|^p (1+\tan^2 R) dR \; .$$

Let us begin with $I$ (ie the $0$-singularity). Once more, we divide the integral in two, between $0$ and $a_n = \arct \frac{1}{n}$ on the one hand, $a_n$ and $\pi/4$ on the other hand.

$$I.1 = \int_{0}^{a_n} |\tan R|^{2-p} |\sin 2n R|^p (1+\tan^2 R) dR$$

Since $|\frac{\sin(2nR)}{\tan R}|^p \leq (2n)^p$, we deduce that : 

$$I.1 \leq (2n)^p \int_{0}^{a_n} \tan^2 R (1+\tan^2 R) dR = (2n)^p \frac{1}{3n^3} = C_p n^{p-3} \; .$$

For $I.2$, we use that the sine is less than $1$, which gives

$$I.2 \leq \int_{a_n}^{\pi/4} |\tan R|^{2-p} (1+\tan^2 R) dR = \int_{1/n}^1 r^{2-p} dr = \frac{1-1/n^{3-p}}{3-p}$$

except when $p=3$, then the majoration is $\log n$. So,

$$I.2 \leq \left \lbrace{ \begin{tabular}{ll}
$C_p$ & \mbox{if} $p< 3$ \\
$C_p \log n$ & \mbox{if} $p=3$ \\
$C_p n^{p-3}$ & \mbox{otherwise} \end{tabular} } \right. $$

Hence : 

$$I \leq \left \lbrace{ \begin{tabular}{ll}
$C_p$ & \mbox{if} $p< 3$ \\
$C_p \log n$ & \mbox{if} $p=3$ \\
$C_p n^{p-3}$ & \mbox{otherwise} \end{tabular} } \right. $$

Let us look at $II$. We do the variable change $R \leftarrow \pi/2 -R$. We have $|\sin(2n (\pi/2 -R))| = |\sin(2n R)| $ and $\tan (\pi/2-R) = \Frac{1}{\tan R}$. The integral $II$ thus is :

$$II= \int_{0}^{\pi/4} |\sin (2nR)|^p |\tan R|^{p-2} (1+\tan^{-2}R) dR = \int_{0}^{\pi/4} |\sin (2nR)|^p |\tan R|^{p-4} (1+\tan^2 R) dR \; .$$

We suppose that $p> \frac{3}{2}$ or the  singularity in $0$ (remember it corresponds to $r=\infty$) diverges.  We get : 

$$II.1 = \int_{0}^{a_n} |\sin (2nR)|^p |\tan R|^{p-4} (1+\tan^2 R) dR \leq (2n)^p \int_{0}^{a_n} |\tan R|^{2p -4} (1+\tan^2 R )dR$$

$$II.1 \leq (2n)^p (\frac{1}{n})^{2p-3}/(2p-3) = C_p n^{3-p}$$

Regarding $II.2$ : 

$$II.2 = \int_{a_n}^{\pi/4} |\sin (2nR)|^p |\tan R|^{p-4} (1+\tan^2 R) dR$$

$$II.2 \leq \int_{a_n}^{\pi/4} |\tan R|^{p-4} (1+\tan^2 R) dR = \int_{1/n}^1 r^{p-4} dr$$

$$II.2  \leq \left \lbrace{\begin{tabular}{ll} 
$C_p$ & \mbox{if} $p> 3$ \\
$C_p \log n$ & \mbox{if} $p=3$ \\
$C_p n^{3-p}$ & \mbox{otherwise} \end{tabular} } \right. $$

Hence

$$II \leq \left \lbrace{\begin{tabular}{ll} 
$C_p$ & \mbox{if} $p> 3$ \\
$C_p \log n$ & \mbox{if} $p=3$ \\
$C_p n^{3-p}$ & \mbox{otherwise} \end{tabular} } \right. $$

Finally, combining $I$ and $II$ gives : 

$$||f_n||_{L^p}\leq \left \lbrace{ \begin{tabular}{lll}
$C_p n^{3/p-1}$ & \mbox{if} $p<3$ \\
$C_p (\log n)^{1/3}$ & \mbox{if} $p=3$ \\
$C_p n^{1-3/p}$ & \mbox{otherwise} \end{tabular} } \right. \; .$$
\end{proof}

\begin{proposition}The initial data $f_0=\sum \frac{h_n}{n}f_n$ belongs to $L^{p}$ almost surely as soon as $p\in ]2,6[$. \end{proposition}

\begin{proof}
The average value of the $L^p_r$ norm to the $p$ of $f_0$ is : 

$$E(||f_0||_{L^p}^p) = \int d\omega \int r^2dr |f_0|^p = \int r^2 dr \int d\omega |f_0|^p = \int r^2 dr ||f_0||_{L^p_\omega}^p\; .$$

We have seen that $||\sum a_n h_n||_{L^p_\omega} \leq C_p \sqrt{\sum |a_n|^2}$, so

$$E(||f_0||_{L^p}^p) \leq \int r^2 dr C_p (\sum |\frac{f_n}{n}|^2)^{p/2} = C_p||\sum |\frac{f_n}{n}|^2||_{L^{p/2}}^{p/2}$$

By a triangle inequality,

$$E(||f_0||_{L^p}^p) \leq C_p(\sum \frac{||f_n||_{L^p}^2}{n^2})^{p/2}$$

$$E(||f_0||_{L^p}^p) \leq \left \lbrace{\begin{tabular}{cccc}
$C_p (\sum n^{-4+6/p})^{p/2}$ & \mbox{if} $p\in ]3/2,3[$ & $< \infty$ & \mbox{if} $p>2$ \\
$C_p (\sum n^{-2}\log^2 n)^{p/2}$ & \mbox{if} $p=3$ & $< \infty$ & \\
$C_p (\sum n^{ -6/p})^{p/2}$ & \mbox{otherwise} & $< \infty $ & \mbox{if} $p<6$
\end{tabular} }\right. $$

Thus, $E(||f_0||_{L^p}^p)$ is finite when $p\in ]2,6[$. We then have that $||f_0||_{L^p}$ is almost surely finite for $p\in ]2,6[$.

\end{proof}

\subsection{Localization}

We are now going to prove the localization of the initial data, that is to say : 

\begin{theorem}The initial data 

$$f_0(\omega,r) = \sum_{n\geq 1} \frac{h_n(\omega)}{n}\frac{\sin (2n\arct(r))}{r}$$
converges toward $0$ when $r\rightarrow \infty$ almost surely in $\omega$.
\end{theorem}

We will also give the behavior of $f_0(\omega,r)$ a.s. in $\omega$ when $r\rightarrow \infty$.

\smallskip

To prove this localization theorem, we are going to use again a variable change.

\begin{lemma} Let $F_0(\omega,R) = \sum_{n\geq 1} \frac{h_n}{n}(-1)^n\sin (2nR)$. If $F_0$ is bounded in a neighborhood of $R=0$ a.s. then $f_0 \rightarrow 0$ when $r\rightarrow \infty$ a.s. .\end{lemma}

\begin{proof}Let $r>0$. Let us compute $F_0(\omega,\frac{\pi}{2}-\arct(r))$.

$$F_0(\omega,\frac{\pi}{2}-\arct(r)) = \sum_{n\geq 1} \frac{h_n}{n} (-1)^n \sin(n\pi -2n\arct(r))$$

$$=-\sum_{n\geq 1} \frac{h_n}{n} \sin(2n\arct(r)) = -rf_0(r)\; .$$

If $F_0(\omega,.)$ is bounded in a neighborhood of $0$ then there exists $R_0>0$ and $M\geq 0$ such that for all $0\leq R\leq R_0$, $|F_0(\omega,R)|\leq M$. Let $r_0 = \tan(\frac{\pi}{2}-R_0) = \frac{1}{\tan R_0}$. Then for all, $r\geq r_0$, we have $\pi/2-\arct(r)\leq R_0$ and so : 

$$|f_0(\omega,r)|= \frac{1}{r}|F_0(\omega, \pi/2-\arct(r))| \leq \frac{M}{r}\rightarrow 0$$
when $r\rightarrow \infty$.
\end{proof}

We are now going to prove some properties about fractional integration of Fourier series. For more information about trigonometrical series in general and fractional integration in particular, one can refer to \cite{zygtwo}, chap. IV.

\begin{definition}Let $c_n$ be a sequence of $\C^{\varmathbb Z}$ such that $c_0 = 0$. We suppose that the distribution $f$ is defined as 

$$f(t) = \sum_{n\in \varmathbb Z} c_n e^{int}\; .$$

We then build $f_\alpha$, where $\alpha$ is a real number such that $0<\alpha< 1$ as :

$$f_\alpha (t) = \sum c_n \frac{e^{int}}{(in)^\alpha}$$
where $(in)^\alpha = e^{\mbox{sign}(n) i\pi \alpha/ 2}|n|^{\alpha}$.
\end{definition}

We show that $f_\alpha$ can be seen as a convolution product.

\begin{definition} We define $\Psi_\alpha (t)$ for $t\in ]0,2\pi[$ as the limit : 

$$\Psi_{\alpha}(t) = \frac{1}{\Gamma(\alpha)}\left( t^{\alpha-1} + \lim_{N\rightarrow \infty} \left(\sum_{n=1}^N (t+2n\pi)^{\alpha-1}-\frac{(2\pi)^{\alpha-1}N^\alpha}{\alpha}\right) \; \right) \; .$$

We also set :

$$r_\alpha^N = \sum_{n=1}^N (t+2n\pi)^{\alpha-1}-\frac{(2\pi)^{\alpha-1}N^\alpha}{\alpha}$$
and

$$r_\alpha = \lim r_\alpha^N \; .$$
\end{definition}

\begin{proposition}The sequences $r_\alpha^N$ and $\frac{dr^N_\alpha}{dt}$ are Cauchy sequences for the $ L^\infty$ norm. Therefore, $r_\alpha$ is a bounded differentiable function whose derivative is bounded. \end{proposition}

\begin{proof} The number $\frac{N^\alpha (2\pi)^{\alpha-1}}{\alpha}$ is equal to : 

$$\frac{N^\alpha (2\pi)^{\alpha-1}}{\alpha} = \frac{1}{2\pi} \int_{0}^{2N\pi} x^{\alpha-1}dx \; .$$

Hence, we get that : 

$$r_\alpha^N (t) = \sum_{n=1}^N \frac{1}{2\pi}\int_{2n\pi}^{2(n+1)\pi}\left( (t+2n\pi)^{\alpha-1}-x^{\alpha-1}\right) dx+\frac{1}{2\pi}\int_{0}^{2\pi}x^{\alpha-1}dx - \frac{1}{2\pi}\int_{2N\pi}^{2(N+1)\pi}x^{\alpha-1}dx$$

$$r_\alpha^N (t) = \sum_{n=1}^N \frac{1}{2\pi}\int_{0}^{2\pi}\left( (t+2n\pi)^{\alpha-1}-(x+2n\pi)^{\alpha-1} \right) dx +\frac{1}{2\pi}\int_{0}^{2\pi}x^{\alpha-1}dx - \frac{1}{2\pi}\int_{2N\pi}^{2(N+1)\pi}x^{\alpha-1}dx$$

Since 

$$|\frac{1}{2\pi}\int_{0}^{2\pi}x^{\alpha-1}dx| = (2\pi)^{\alpha-1} < \infty$$
and

$$|\frac{1}{2\pi}\int_{2N\pi}^{2(N+1)\pi}x^{\alpha-1}dx| = \frac{(2\pi)^{\alpha-1}}{\alpha}|(N+1)^{\alpha}-N^\alpha|\leq CN^{\alpha-1} \rightarrow 0$$
and also

$$|(t+2n\pi)^{\alpha-1}-(x+2n\pi)^{\alpha-1}|\leq \frac{1}{(2n\pi)^{2-\alpha}}|x-t|\; ,$$
we have that

$$|r_N^\alpha(t)| \leq C + \sum_{n=1}^N \frac{1}{(2\pi n)^{2-\alpha}}\int_{0}^{2\pi}|x-t|dx$$
and

$$|r_N^\alpha(t)-r_M^\alpha(t)| \leq C(N^{\alpha-1}+M^{\alpha-1} + \sum_{n=M+1}^N \frac{1}{(2\pi n)^{2-\alpha}})\; .$$

as $2-\alpha > 1$ and $|x-t| \leq 4\pi$ we get that $r_N^\alpha(t)$ is bounded by a constant independant from $t$ and from $N$ and that $r_N^\alpha$ is a Cauchy sequence for the $L^\infty $ norm. Then, as the $r_N^\alpha$ are continuous, $r_\alpha$ is also continuous, and it is bounded. 

\smallskip

The derivative $\frac{d r^N_\alpha}{dt}$ is

$$\frac{d r^N_\alpha}{dt} = (\alpha-1) \sum_{n=1}^N (t+2n\pi)^{\alpha-2}$$

The $L^\infty$ norm of $(t+2n\pi)^{\alpha -2 }$ is $(2n\pi)^{\alpha -2}$ which is the general term of a convergent series. Then $\frac{d r_\alpha}{dt}$ is well defined, continuous and bounded.
\end{proof}

\begin{proposition}The Fourier coefficients of $\Psi_\alpha(t)$ are $c_0 = 0$ and $c_n = \frac{1}{(in)^\alpha}$.\end{proposition}

\begin{proof} As $r^N_\alpha$ converges uniformally and is $||\; .\; ||_{L^\infty}$-bounded, we can swap the integral and the limit ($t^{\alpha -1}$ is integrable). Then,

$$\Gamma(\alpha) c_0 = \lim \frac{1}{2\pi}\int_0^{2\pi} (t^{\alpha-1}+\sum_{n=1}^N (t+2n\pi)^{\alpha-1}) -\frac{1}{2\pi}\int_{0}^{2N\pi} x^{\alpha-1} dx )$$

$$ = \lim \frac{1}{2\pi} \int_{2N}^{2(N+1)\pi} t^{\alpha-1}dt $$

$$= 0$$
and for $n\neq 0$,

$$\Gamma(\alpha) c_n = \lim \frac{1}{2\pi}\int_0^{2\pi} e^{-int}(t^{\alpha-1}+\sum_{n=1}^N (t+2n\pi)^{\alpha-1}) -\frac{1}{2\pi}\int_{0}^{2\pi} (2N\pi)^\alpha e^{-int}dt $$

Since $\int_{0}^{2\pi} (2N\pi)^\alpha e^{-int}dt  =0$ we have : 

$$\Gamma(\alpha) c_n = \int_{0}^{\infty} t^{\alpha-1} e^{-int}dt\; . $$

The function $z\mapsto z^{\alpha-1}e^{-|n|z}$ is analytic on $\C \smallsetminus \{ 0\} $. We integrate it on the axis $[\epsilon,R]$ and $[\mbox{sign}(n) i \epsilon, \mbox{sign}(n) i R]$ and the arcs of radius $\epsilon$ and $R$. Then, we do $\epsilon \rightarrow 0$ and $ R\rightarrow \infty$. The integral over the arc of radius $\epsilon$ behaves like $\epsilon^\alpha$ when $\epsilon \rightarrow 0$ so it converges to $0$. For $R$, the behaviour is given by $e^{-cR}R^\alpha$. We deduce from that : 

$$\Gamma(\alpha) c_n = e^{-i\mbox{sign}(n)\alpha\pi/2}\int_{0}^\infty e^{-|n|t}t^{\alpha-1}dt$$

With a variable change $u=|n|t$, we get

$$\Gamma(\alpha) c_n = e^{-i\mbox{sign}(n)\alpha/2}|n|^{-\alpha} \int_{0}^\infty e^{-u}u^{\alpha-1} dt = (in)^{-\alpha} \Gamma(\alpha)\; .$$

\end{proof}

A more detailed proof can be found in \cite{zygone} chap. II.

\begin{proposition}Let $\alpha\in ]0,1[$ and $r$ such that $\alpha r > 1$. Let $f = \sum c_n(\omega) e^{int} \in L^r(\Omega\times [0,2\pi])$ with $c_0 =0$ and suppose that $f$ is the $L^r$ limit of $\sum_{n=-N}^N c_n e^{int}$. Then 

$$\int_{0}^{2\pi}\frac{dt}{2\pi} f(\omega,t) \Psi_\alpha(x-t)$$ 
is in $L^r_\omega,L^\infty_x$ and is the limit of $\sum_{n=-N}^N \frac{c_n}{(in)^\alpha}e^{inx}$ that is to say $f_\alpha(\omega,x)$.
\end{proposition}

\begin{proof}As $\alpha r >1$, we get that $\alpha -1 > \frac{1}{r}-1 = -\frac{1}{r'}$ where $r'$ is the conjugate number of $r$. Then, as $|\Psi_\alpha (t)| \leq C(1+t^{\alpha-1})$, we have that $\Psi_\alpha$ is in $L^{r'}$. This way, we have : 

$$||\frac{1}{2\pi}\int_{0}^{2\pi} f(t)\Psi_\alpha(x-t)dt||_{L^\infty_x} \leq ||f||_{L^r_x}||\Psi_\alpha||_{L^{r'}}$$
that is to say

$$||\frac{1}{2\pi}\int_{0}^{2\pi} f(t)\Psi_\alpha(x-t)dt||_{L^r_\omega,L^\infty_x}\leq ||\Psi_\alpha||_{L^{r'}_x}||f||_{L^r_{\omega,x}}< \infty \; .$$

What's more, 

$$\frac{1}{2\pi}\int_{0}^{2\pi} c_n e^{int} \Psi_\alpha(x-t) dt = c_n e^{inx}\frac{1}{2\pi}\int_{0}^{2\pi}e^{-inu}\Psi_\alpha(u) du$$

$$=\frac{c_n}{(in)^\alpha} e^{inx}$$

Then,

$$\frac{1}{2\pi}\int_{0}^{2\pi} f(t)\Psi_\alpha(x-t)dt-\sum_{n=-N}^N \frac{c_n}{(in)^\alpha}e^{inx} =
\frac{1}{2\pi}\int_{0}^{2\pi} (f(t)-\sum_{n=-N}^N c_n e^{int})\Psi_\alpha(x-t) dt$$

$$||\frac{1}{2\pi}\int_{0}^{2\pi} f(t)\Psi_\alpha(x-t)dt-\sum_{n=-N}^N \frac{c_n}{(in)^\alpha}e^{inx}||_{L^\infty_x} \leq
||f(t)-\sum_{n=-N}^N c_n e^{int}||_{L^r_t}||\Psi_\alpha||_{L^{r'}}$$

$$||\frac{1}{2\pi}\int_{0}^{2\pi} f(t)\Psi_\alpha(x-t)dt-\sum_{n=-N}^N \frac{c_n}{(in)^\alpha}e^{inx}||_{L^r_\omega,L^{\infty}_x}\leq ||f(t)-\sum_{n=-N}^N c_n e^{int}||_{L^r_\omega,L^r_t}||\Psi_\alpha||_{L^{r'}}$$
which converges towards $0$ by hypothesis.

\end{proof}

\begin{theorem}Under the same assumptions, we have that $f_\alpha(x+h)-f_\alpha(x)$ is a.s. a $O(h^\nu)$ where $\nu =(\alpha  -1/r) >0$. \end{theorem}

\begin{proof} Once more, we can write : 

$$|f_\alpha(x+h)-f_\alpha(x)| \leq ||f||_{L^r_t}\left(\frac{1}{2\pi}\int_{0}^{2\pi}|\Psi_\alpha(u+h)-\Psi_\alpha(u)|^{r'}du\right)^{1/r'}$$

We the divide the integral over $u$ into two parts : one from $0$ to $h$, the other from $h$ to $2\pi$.

$$\int_{0}^{h}|\Psi_\alpha(u+h)-\Psi_\alpha(u)|^{r'}du\leq C \int_{0}^{2h}|\Psi_\alpha(u)|^{r'}du$$

$$\leq C \int_{0}^{2h}|u|^{(\alpha-1)r'}du$$

$$\leq C h^{r'(\alpha-1/r)}$$

$$\int_{h}^{2\pi}|\Psi_\alpha(u+h)-\Psi_\alpha(u)|^{r'}du\leq h^{r'} \int_{h}^{2\pi} u^{(\alpha-2)r'}$$

$$\leq C h^{r'}h^{r'(\alpha-2)+1} = C h^{r'(\alpha -1/r)}$$

We have then

$$h^{-\nu} |f_\alpha(x+h)-f_\alpha(x)| \leq C ||f||_{L^r_t}$$

$$||h^{-\nu} |f_\alpha(x+h)-f_\alpha(x)||_{L^r_\omega,L^\infty_{x,h}}\leq C ||f||_{L^r_{\omega,t}}<\infty$$

We deduce from that that $||h^{-\nu} |f_\alpha(x+h)-f_\alpha(x)||_{L^\infty}$ is a.s. finite, that is to say that $f_\alpha(x+h)-f_\alpha(x)$ is a.s. a $O(h^\nu)$.
\end{proof}

\begin{definition} Let $0<\alpha<\frac{1}{2}$, $r> \frac{1}{\alpha} > 2$ and $F_{-\alpha} = \sum_{n\neq 0}(-1)^n \frac{h_{|n|}}{2in}(in)^\alpha e^{2inx}$ defined as the limit of $ \sum_{n = -N}^N (-1)^n\frac{h_{|n|}}{2in}(in)^\alpha e^{2inx}$ in $L^r_{\omega,x}$.\end{definition}

We have to prove that the sequence converges.

\begin{proof} 

$$||\sum_{n = -N}^N (-1)^n\frac{h_{|n|}}{2in}(2in)^\alpha e^{2inx} - \sum_{n = -M}^M (-1)^n\frac{h_{|n|}}{2in}(in)^\alpha e^{2inx}||_{L^r}^r $$

$$\leq C\int_{0}^{2\pi} ||\sum_{n=M+1}^N (-1)^n\frac{h_n}{n^{1-\alpha}}\frac{e^{inx}i^\alpha- e^{-inx}(-i)^{\alpha}}{2}||_{L^r_\omega}^r dx$$

Since $r>2 $ we have that $||\sum c_n h_n||_{L^r_\omega}\leq C_r (\sum |c_n|^2)^{1/2}$, then,

$$\leq C_r \int_{0}^{2\pi} (\sum_{n=M+1}^N \frac{1}{n^{2-2\alpha}})^{r/2}$$

As $\alpha <\frac{1}{2}$, $2-2\alpha > 1$, so the series of general term $n^{2\alpha-2}$ converges and the sequence $\sum_{n=-N}^N \frac{h_{|n|}}{2in}(2in)^\alpha e^{2inx}$ is a Cauchy sequence in $L^r_{\omega,x}$, it converges.
\end{proof}

\begin{lemma}Let $\nu < \frac{1}{2}$, there exists $\alpha<\frac{1}{2}$ and $r>\frac{1}{\alpha}$, such that $\nu = \alpha-\frac{1}{r}$ and $(F_{-\alpha})_\alpha = F_0$ in $L^r_\omega,L^\infty_x$ and so $F_0(h)$ is almost surely in $\Omega$ a $O(h^\nu)$.\end{lemma}

\begin{proof}Set $\nu < \alpha < \frac{1}{2}$ and let $r = \frac{1}{\alpha-\nu} > \frac{1}{\alpha}$. The function $F_{-\alpha} $ is the limit in $L^r_{\omega,x}$ of the sequence $\sum_{n=-N}^N (-1)^n\frac{h_{|n|}}{2in} (in)^\alpha e^{inx}$, so $(F_{-\alpha})_\alpha$ is the limit in $L^r_\omega,L^\infty_x$ of 

$$ \sum_{n=-N}^N (-1)^n\frac{h_{|n|}}{2in}\frac{(2in)^\alpha}{(2in)^\alpha}e^{2inx} = \sum_{n=1}^N \frac{h_n}{n}\frac{e^{2inx}-e^{-inx}}{2i}(-1)^n$$

$$ = \sum_{n=1}^N \frac{h_n}{n}\sin(2nx)(-1)^n \; .$$

That is to say, $F_0 = (F_{-\alpha})_\alpha$ in $L^r_\omega,L^\infty_x$ and so we have that $F_0(x+h)-F_0(x)$ is almost surely in $\Omega$ a $O(h^\nu)$ as $\nu = \alpha-\frac{1}{r}$. We use it for $x=0$. We have $F_0(\omega,0) = 0$ so a.s.

$$F_0(\omega,x) = O(x^\nu) \; .$$

\end{proof}

\begin{theorem}Almost surely, when $r\rightarrow \infty$, $f_0(\omega,r)$ is a $O(\frac{1}{r^{1+\nu}})$ for all $\nu <\frac{1}{2}$.\end{theorem}

\begin{proof} We have 

$$f_0(\omega,r) = \frac{1}{r}F_0(\omega,\pi/2-\arct(r))\; .$$

Indeed, $\pi/2 -\arct(r) = \arct(\frac{1}{r}) = O(\frac{1}{r})$. Then, since $F_0(\omega,h) = O(h^\nu)$ a.s., we have that a.s. : 

$$f_0(\omega,r) = \frac{1}{r} O(\arct^\nu(\frac{1}{r})) = O(\frac{1}{r^{1+\nu}}) \; .$$
\end{proof}

\subsection{$L^p$-Spaces the initial data do not belong to}
 
Now, we are going to see that for $p\notin ]2,6[$, $||f_0||_{L^p}=\infty $ almost surely.

\medskip

First step. 

\begin{proposition} For almost all $r\in \R$, the function $f_0(.,r) = \sum \Frac{h_n}{n} f_n(r)$ is a real gaussian variable of variance $\sigma^2(r) = \sum \Frac{|f_n(r)|^2}{n^2}$.
\end{proposition}

\begin{proof}

Remark that $\sigma(r)$ is a.s. finite since 

$$\int r^2 dr \sigma^4(r) = \sum \frac{||f_n||_{L^2}^4}{n^4}\leq C \sum n^{-2}< \infty$$

Let's compute $E(e^{ix f_0})$ when $\sigma(r)<\infty$.

$$E(e^{ixf_0}) = E(\prod e^{ix \frac{h_n f_n}{n}}) = \prod E(e^{ix \frac{f_n}{n} h_n})$$
as the $h_n$ are independant. Since $h_n$ is a real centered gaussian, we have that  $E(e^{ix \frac{f_n}{n} g_n}) = e^{-x^2 \frac{|f_n(r)|^2}{2n^2}}$ hence : 

$$E(e^{ix f_0}) = e^{-x^2 \sum \frac{|f_n^2(r)|}{2n^2}} = e^{-x^2 \sigma^2(r)/2}$$
so $f_0(r)$ is a real gaussian variable of variance $\sigma^2(r)$.

\smallskip

\end{proof}

\begin{lemma} For all real centered gaussian $Z$,

$$E(|Z|^p) = C(p) E(|Z|^2)^{p/2}$$
\end{lemma}

\begin{proof} Indeed,

$$E(|Z|^p) = \int |Z|^p e^{-\frac{|Z|^2}{2\sigma^2}}\frac{dZ}{\sqrt{2\pi} \sigma}$$

$$=\int \sigma^p |Y|^p e^{-|Y|^2/2}\frac{dY}{\sqrt{2\pi}} = E(|Z|^2)^{p/2} \int |Y|^pe^{-|Y|^2}\frac{dY}{\sqrt{2\pi}}$$

\end{proof}

\smallskip

We deduce from that

$$E(||f_0||_{L^p}^p) = \int r^2dr E(|f_0|^p) = C(p) \int r^2 dr E(|f_0|^2)^{p/2} = C\int r^2 dr (\sum \frac{|f_n|^2}{n^2})^{p/2}$$

\begin{lemma}Let $E$ be a vector space and $\mathcal B(E)$ its borel $\sigma$-algebra. We set $F$ a centered gaussian variable on $E$ and $N$ a pseudo-norm on $E$ that is to say a norm that admits $+\infty$ as a possible value. If the probability $P(N(F)<\infty)$ is strictly positive, then for all $p<\infty$, $E(N^p(F))<\infty$. \end{lemma}

The proof of this lemma is given by X. Fernique in \cite{fernique}.

\begin{corollary}Under the previous assumptions, if there exists $p$ such that $E(N^p(F))=\infty$, then $P(N(F)<\infty)=0$.\end{corollary}

In particular with $N = ||.||_{L^p}$, we get that if $E(||f_0||_{L^p}^p) = \infty$, then $||f_0||_{L^p}$ is a.s. infinite.

\smallskip

We have to study $E(||f_0||_{L^p}^p)$, that is $\int r^2 dr (\sum \Frac{|f_n|^2}{n^2})^{p/2}$.

\begin{lemma}Let $R_n = \Frac{\pi}{4n}$. For all $R\in[0,R_n]$, $|f_n(\tan R)|\geq n$. \end{lemma}

\begin{proof}
We have : 

$$f_n(\tan R) =\frac{\sin(2n R)}{\tan R}$$

For all $0\leq R \leq \frac{\pi}{4}$, $|\tan R|\leq \frac{4}{\pi}R$, and for all $0\leq R \leq \frac{\pi}{2}$, $|\sin R|\geq \frac{2}{\pi}R$ then for all $0\leq R \leq R_n$, 

$$|f_n(\tan R)| \geq \frac{4n R}{\pi}\frac{\pi}{4R} = n$$

\end{proof}

\smallskip

We have immediately that

$$\sum \frac{|f_n(\tan R)|^2}{n^2} \geq \sum 1_{R\leq R_n} = \sum 1_{R_{n+1}\leq R \leq R_n} n$$

By changing $r$ into $\tan R$, we get : 

$$E(||f_0||_{L^p}^p)\geq \int_{0}^{\pi/2} \tan^2 R (1+\tan^2R) dR (\sum \frac{|f_n(\tan R)|^2}{n^2})^{p/2}$$

We divide the integral into its computation $I$ on $[0,\pi/4]$ and $II$ on $[\pi/4,\pi/2]$. For $II$, we change $R$ into $ \pi/2 -R$, we get : 

$$II = \int_{0}^{\pi/4} \tan^{-4}R(1+\tan^2 R)dR (\sum \frac{|f_n(\tan R)|^2 \tan^4 R}{n^2})^{p/2}$$

$$=\int_{0}^{\pi/4} \tan^{4(p/2-1)}R (1+\tan^2 R)dR (\sum \frac{|f_n(\tan R)|^2}{n^2})^{p/2}$$

For $p\geq 6$, we use a minoration of $I$ ($0$ singularity).

$$I\geq \int_{0}^{\pi/4} \tan^2R(1+\tan^2 R) dR \sum n^{p/2} 1_{R_{n+1}\leq R \leq R_n}=\sum \frac{r_n^3-r_{n+1}^3}{3}n^{p/2}$$
with $r_n = \tan R_n$. We have $r_n^3-r_{n+1}^3 \sim 3\Frac{4}{\pi} \Frac{1}{n^4}$, so the general term of the series behaves like $n^{p/2-4}$. For $p\geq 6$, $p/2-4 \geq -1$ so the series diverges.

\bigskip

For $p\leq 2$, we use a minoration of $II$ ($\infty$ singularity).

$$II\geq \sum n^{p/2}\int_{R_{n+1}}^{R_n} \tan^{2p-4}R(1+\tan^2 R) dR = \sum n^{p/2} \frac{r_n^{2p-3}-r_{n+1}^{2p-3}}{2p-3}$$

$r_n^{2p-3}-r_{n+1}^{2p-3}$ behaves like $n^{2-2p}$ so the general term behaves like $n^{-3p/2+2}$. For $p\leq 2$, $-3p/2+2\geq -1$ so the series diverges. We get thet for $p\leq 2$ or $p\geq 6$, $E(||f_0||_{L^p}^p)= \infty$ and so $||f_0||_{L^p}=\infty $ almost surely.

\smallskip

Conclusion

\begin{proposition} The initial data $f_0$ is almost surely in $L^p$ for $p\in ]2,6[$ and almost surely outside $L^p$ for $p\leq 2$ and $p\geq 6$.\end{proposition}

\subsection{Regularity of $f_1$}

We have that $f_1 = \sum l_n \Frac{2}{1+r^2} f_n $. Let us do a change of variable $r\mapsto R = \arct(r)$ and use the result and methods we developped in the localization section. More precisely, we are going to study the behaviour of a ``primitive" of $f_1$ by using the change of variable and use the study of behaviour of periodic functions at two given points $R=0$ corresponding to $r=0$ and $R=\Frac{\pi}{2}$ corresponding to $r=\infty$.

\begin{definition} Let $R_0 < \frac{\pi}{4}$, $V_1= [0,\frac{\pi}{2}-R_0]$ and $V_2 = [R_0, \frac{\pi}{2}]$ Set $\psi_n^1(R) = 1_{V_1} \frac{1}{n}(1-\cos (2nR))$ and $\psi_n^2 (R) = 1_{V_2}\frac{1}{n} ((-1)^n-\cos(2nR))$. Then, let $\Psi_\alpha^1$ and $\Psi_\alpha^2$ be the $L^r$ limits for $\alpha<\frac{1}{2}$ and $r>2$ of 

$$\sum_{n>0} (in)^\alpha l_n \psi_n^j$$
with $j=1,2$.

This limit exists and hence we can define 

$$\Psi^j = \sum_{n>0} l_n \psi_n^j$$
as the limits of the partial sums in $L^r_\omega,L^\infty_R$ and have the properties : 

$$\Psi^1(R) = O(R^\nu)$$
when $R\rightarrow 0$ for all $0<\nu < \frac{1}{2}$ almost surely in $\omega$ and

$$\Psi^2(\frac{\pi}{2}-R) = O(R^\nu)$$
when $R\rightarrow 0$ for all $0<\nu<\frac{1}{2}$ a.s. in $\omega$.
\end{definition}

\begin{proposition}Let  $\Phi^j (r)=\frac{\Psi^j(\arct(r))}{r}$ for $j=1,2$. We have that $\Phi^j$ is almost surely in $L^p$ for all $p\in ]2,6[$.\end{proposition}

\begin{proof}Let $p\in ]2,6[$ and let us compute the $L^p$ norm of $\Phi^j$.

$$||\Phi^j||_{L^p}^p = \int_{0}^\infty |\Phi^j (r)|^p r^2 dr $$

$$ = \int_{0}^{\pi/2} |\Psi^j(R)|^p (1+\tan^2 R) |\tan R|^{2-p}dR \; .$$

Considering the different $j$, we get : 

$$ ||\Phi^1||_{L^p}^p= \int_{V_1} |\Psi^1(R)|^p (1+\tan^2 R) |\tan R|^{2-p}dR$$
and

$$||\Phi^2||_{L^p}^p = \int_{V_2}|\Psi^2(R)|^p (1+\tan^2 R) |\tan R|^{2-p}dR$$
that is to say, with the variable change $R\leftarrow \frac{\pi}{2}-R$ :

$$|\Phi^2||_{L^p}^p = \int_{V_1}|\Psi^2(\pi/2-R)|^p (1+\tan^2 R) |\tan R|^{p-4}dR \; .$$

Since $1-\Frac{3}{p} < \Frac{1}{2}$ as $p<6$ we can choose $\nu_1 \in ]1-\Frac{3}{p},\Frac{1}{2}[$ that is to say $2-p+p\nu_1 > -1$. Then, given that $\Psi^1(R) = O(R^{\nu_1})$, we get that

$$|\Psi^1(R)|^p |\tan R|^{2-p} = O(R^{\nu_1 p+2-p})$$
so the first integral $||\Phi^1||_{L^p}^p$ converges almost surely in $\omega$.

\smallskip

Since $\Frac{3}{p}-1 < \Frac{1}{2}$ as $p>2$, we can choose $\nu_2 \in ]\Frac{3}{p}-1,\Frac{1}{2}[$, that is to say $p-4+p\nu_2 > -1$. Then, given that $\Psi^2(\pi/2-R) = O(R^{\nu_2})$, we get that

$$|\Psi^2(\pi/2-R)|^p |\tan R|^{p-4} = O(R^{\nu_2 p +p-4})$$
so the second integral $||\Phi^2||_{L^p}^p$ converges a.s. 

\smallskip

Hence, $\Phi$ belongs to $L^p$ for all $p\in ]2,6[$ almost surely. \end{proof}

\begin{definition}Let $\chi_1$ and $\chi_2$ be two $\mathcal C^\infty_c$ functions with supports included in $V_1$ and $V_2$ respectively such that $\chi_j \in [0,1]$ and $\chi_1 + \chi_2 =1$ on $[0,\frac{\pi}{2}]$. We set 

$$\Phi(r) = \chi_1(\arct(r)) \Phi_1(r) + \chi_2(\arct(r)) \Phi_2(r)$$
and we call $F$ the distibution defined as :

$$F(r) = \Phi'(r) + \frac{\Phi(r)}{r} - \frac{1}{(1+r^2)}\left(\chi_1'(\arct(r))\Phi^1(r) + \chi_2'(\arct(r))\Phi^2(r)\right) \; .$$
\end{definition}

\begin{proposition}We deduce from the precedent proposition that $F$ belongs to $W^{-1,p}$ for all $p\in ]2,6[$.\end{proposition}

\begin{proof}Indeed, $\Phi'$ belongs to $W^{-1,p}$ by using a simple derivative.

\smallskip

Then we use Sobolev embedding theorem. Set $q$ such that $\Frac{1}{q} = \Frac{1}{p}+\Frac{s}{3}$, there exists $C$ such that for all $f$,

$$||f||_{L^p} \leq C ||(1-\Delta_{\R^3})^{s/2} f ||_{L^q} \; .$$

We use Sobolev embedding theorem with $s=1$ and 

$$f=(1-\Delta)^{-1/2}\left( \frac{\Phi}{r}- \frac{1}{(1+r^2)}\left(\chi_1'(\arct(r))\Phi^1(r) + \chi_2'(\arct(r))\Phi^2(r)\right)\right) \; ,$$
we get that

$$||\left( \frac{\Phi}{r}- \frac{1}{(1+r^2)}\left(\chi_1'(\arct(r))\Phi^1(r) + \chi_2'(\arct(r))\Phi^2(r)\right)\right) ||_{W^{-1,p}}$$

$$\leq ||\left( \frac{\Phi}{r}- \frac{1}{(1+r^2)}\left(\chi_1'(\arct(r))\Phi^1(r) + \chi_2'(\arct(r))\Phi^2(r)\right)\right)||_{L^q} \; .$$

We then have to prove that $\Frac{\Phi}{r}$, $\frac{1}{1+r^2}\chi'_j(\arct(r))\Phi^j(r)$ for $j=1,2$ belong to $L^q$.

\smallskip

Now, let us remark that $||\Frac{\Phi}{r}||_{L^q}\leq ||\frac{\Phi^1}{r}||_{L^q}+||\frac{\Phi^2}{r}||_{L^q}$.

\smallskip

Let $p_1\in ]p,6[$ and $p_2\in ]2,p[$ and $p'_j$ defined as $\Frac{1}{p}+\Frac{1}{3}=\Frac{1}{q} = \Frac{1}{p_j}+\Frac{1}{p'_j}$. We have : 

$$\frac{1}{p'_1} = \frac{1}{3}+\frac{1}{p}-\frac{1}{p_1}\in]\frac{1}{3}, \frac{1}{2}+\frac{1}{p}[$$

So, $p'_1 \in ]1,3[$ and by a similar computation $p'_2 \in ]3,\infty[$. Hence, 

$$||\frac{\Phi^1}{r}||_{L^q} \leq ||\frac{1_{r<\arct(\pi/2-R_0)}}{r}||_{p'_1}||\Phi^1||_{L^{p_1}} < \infty $$
and

$$||\frac{\chi_1'(\arct(r)) \Phi^1}{1+r^2}||_{L^q} \leq ||\frac{\chi'(\arct(r))}{1+r^2}||_{L^{p'_1}}||\Phi^1||_{L^p_1}$$
and since $\chi'_1(\arct(r))$ is bounded and null outside a bounded set of $\R$ we get that : 

$$||\frac{\chi_1'(\arct(r)) \Phi^1}{1+r^2}||_{L^q}<\infty \; .$$

Also, 

$$||\frac{\Phi^2}{r}||_{L^q} \leq ||\frac{1_{r>\arct(R_0)}}{r}||_{p'_2}||\Phi^2||_{L^{p_2}} < \infty $$
and

$$||\frac{\chi_2'(\arct(r)) \Phi^2}{1+r^2}||_{L^q} \leq ||\frac{\chi'(\arct(r))}{1+r^2}||_{L^{p'_2}}||\Phi^2||_{L^p_2}$$
since $r\mapsto \chi'_2(\arct(r))$ is bounded, we get 

$$||\frac{\chi_2'(\arct(r)) \Phi^2}{1+r^2}||_{L^q} < \infty\; .$$

Thus, $F$ belongs to $W^{-1,p}$ for all $p\in ]2,6[$. \end{proof}

\begin{lemma} The distribution $F$ is equal to the initial data $f_1 = \sum_{n\neq 0} l_n \frac{\sin(2n\arct(r))}{r(1+r^2)}$. \end{lemma}

\begin{proof} We have that $F = \sum l_n(\omega) F_n(r)$ with 

$$F_n (r) = \left(\frac{d}{dr}+ \frac{1}{r}\right)\left(\sum_{j=1,2}\frac{\chi_j(\arct(r))((-1)^{n(j+1)}-\cos(2n\arct(r)))}{2nr}\right)$$

$$ - \frac{1}{1+r^2}\left( \frac{\sum_{j=1,2} \chi'_j(\arct(r))((-1)^{n(j+1)}-\cos(2n\arct(r)))}{2nr}\right)$$

Let us compute $F_n$.

$$\left(\frac{d}{dr}+ \frac{1}{r}\right)\left(\sum_{j=1,2}\frac{\chi_j(\arct(r))((-1)^{n(j+1)}-\cos(2n\arct(r)))}{2nr}\right) = $$

$$\frac{d}{2nrdr}\left(\sum_{j=1,2}\chi_j(\arct(r))((-1)^{n(j+1)}-\cos(2n\arct(r)))\right)\; .$$

So,

$$F_n = \sum_{j=1,2} \frac{\chi_j(\arct(r))}{2nr}\frac{d}{dr}\left( (-1)^{n(j+1)}-\cos(2n\arct(r))\right)$$

$$F_n = (\chi_1(\arct r) + \chi_2(\arct(r)) \frac{1}{1+r^2} \frac{\sin(2n\arct(r))}{r}\; .$$

Hence, we get that the initial data 

$$f_1(r) = \sum_{n>0} l_n \frac{\sin(2n\arct(r))}{r(1+r^2)} = \sum l_n F_n = F \; ,$$
and we deduce from that the initial data is almost surely in $W^{-1,p}$. \end{proof}

\begin{theorem}The initial data $f_0,f_1$ is almost surely in $L^p \times W^{-1,p}$ for all $p\in ]2,6[$. \end{theorem}

\subsection{Consequences on the regularity of the solution}

We have seen in the definition of the global flow of \reff{nlwe} that the solution is of the form $L(t)(f_0,f_1) + g(t,.)$ where $L(t)$ is the flow of the free wave equation, $f_0,f_1$ is the initial data, and $g$ belongs to $L^p_{(t,r)}$. Thanks to the previous considerations, we will see that $L(t)(f_0,f_1)$, as it highly resembles in spatial structure the initial data $f_0,f_1$, is localized almost surely. Furthermore, the ``controlled" part of the solution $g$, admits some kind of localization.

\begin{proposition} Let $f_0,f_1 \in \Pi $ and write the solution of the non linear wave equation \reff{nlwe} with initial data $f_0,f_1$ 

$$f(t,r) = L(t)(f_0,f_1)(r) + g(t,r) \; .$$

We have that : 

\begin{itemize}

\item $L(t)(f_0,f_1)$ is $\omega$ - almost surely localized, that is to say, for almost all $f_0,f_1 \in \Pi $, and all $t\in \R$, 

$$\lim_{r\rightarrow \infty} L(t)(f_0,f_1) (r) = 0 \; ,$$

\item for $p\in ]2\alpha, 6[$ such that the solution is defined in $L(t)(f_0,f_1) + L^p_{t,r}$, (see theorem \reff{globsol}) we have that $\left(\frac{1+r^2}{2}\right)^{1/2-2/p} g(t,r)$ is in $L^p_{t,r}$, hence, for almost all $t\in \R$, $\left(\frac{1+r^2}{2}\right)^{1/2-2/p} g(t,r)$ is in $L^p_r$.

\end{itemize}
\end{proposition}

\begin{proof}Let us prove the second point of the proposition. We have seen in the proof of \reff{globsol} that for all radial function $\psi $, and its global Penrose transform $u$ that :

$$||\psi ||_{L^p_{t,r}} = ||\Omega^{1-4/p} \Re u ||_{L^p_{T,R}} \; .$$

But if we write $u$ the global Penrose transform of the solution $f$. We have that $u=S(T)u_0 + v$ with $u_0$ the Penrose tranform at time $T=0$ of $f_0,f_1$, $S(T)$ the flow of the linear wave equation on the sphere $S^3$ and $v\in X^s_{2\pi}$. So, as the global Penrose transform turns $S(T)$ into $L(t)$, and that it is a linear application, we get that $g$ is the global Penrose transform of $v$. Thus, as 

$$\sqrt{\frac{1+r^2}{2}} \leq \left \lbrace{ \begin{tabular}{ll} 
$\sqrt{\frac{1+(t+r)^2}{2}}$ & \mbox{if} $t\geq 0$ \\
$\sqrt{\frac{1+(t-r)^2}{2}}$ & \mbox{otherwise}\end{tabular} } \right. \; ,$$
we have

$$\sqrt{\frac{1+r^2}{2}} \leq \sqrt{\frac{(1+(r+t)^2)(1+(r-t)^2)}{2}} \leq \sqrt 2 \Omega(t,r)^{-1} $$
and so

$$||\left(\frac{1+r^2}{2}\right)^{1/2-2/4} g ||_{L^p_{t,r}}\leq C ||\Omega^{4/p -1 }g||_{L^p_{t,r}} \leq C ||v||_{L^p_{T,R}}$$

$$||\left(\frac{1+r^2}{2}\right)^{1/2-2/4} g ||_{L^p_{t,r}}\leq C ||v||_{X^s_{2\pi}}\; .$$

Let us now prove the first part of the proposition. We set $l(t,r) = L(t)(f_0,f_1)(r)$. This function is the global Penrose transform of $S(T)u_0$. Hence, it can be writen under the form :

$$l(t,r) = \frac{2\left((1+(t+r)^2)(1+(t-r)^2)\right)^{-1/2}}{\sin(\arct(t+r)-\arct(t-r))} \times $$

$$ \sum_n \frac{h_n}{n} e^{-in(\arct(t+r)+\arct(t-r))} \sin(2n(\arct(t+r)-\arct(t-r))) \; .$$

The factor $\frac{2}{\sqrt{(1+(t+r)^2)(1+(t-r)^2)}\sin(\arct(t+r)-\arct(t-r))} $ is equal to $\frac{1}{r}$ as we have seen in the preliminaries, hence it remains to show that the sum is bounded. We will divide the sum in two parts as (consider $t$ as fixed) : 

$$l(t,r) = L_1(Y) + L_2(Z)$$
with

$$2Y(r) = \arct(t+r) - 3 \arct(t-r) \; , \; 2Z(r) = 3\arct(t+r) - \arct(t-r)$$
and

$$L_1(Y) = \sum_{n\geq 1} \frac{h_n}{2in}e^{2in Y} \; , \; L_2(Z) = \sum_{n\leq -1} \frac{h_{-n}}{2in} e^{2inZ} \; .$$ 

Let $0< \alpha < \frac{1}{2}$ and set

$$L_1^\alpha (Y) = \sum_{n\geq 1} \frac{h_n}{2in}(2in)^\alpha e^{2in Y} \; \mbox{and} \; L_2^\alpha (Z) = \sum_{n\leq -1} \frac{h_{-n}}{2in}(2in)^\alpha e^{2in Z} \; .$$

The sum $L_1^\alpha $ (resp. $L_2^\alpha$) is the limit in $L^r_{\omega,Y}$ (resp. $L^r_{\omega,Z}$) of the partial sum

$$\sum_{n=1}^N \frac{h_n}{2in}(2in)^\alpha e^{2in Y} \mbox{ (resp. } \sum_{n= -N}^{-1} \frac{h_{-n}}{2in}(2in)^\alpha e^{2in Z} \; )$$
for all $r\geq 2$. Hence the sums $L_1$ and $L_2$ are limits of finite sums in $L^r_\omega, L^\infty_Y$ or $L^r_\omega,L^\infty_Z$ for all $r$ such that $\alpha r > 1$ and thus are $\omega$-almost surely bounded in $Y,Z$.

\smallskip

We deduce from that that for all $t\in \R$ fixed, $r\in \R_+^*$ and almost all $\omega$, $L(t)(f_0^\omega,f_1^\omega)$ satisfies : 

$$|L(t)(f_0^\omega,f_1^\omega)(r)| \leq \frac{1}{r} \left(||L_1(\omega)||_{L^\infty_Y} + ||L_2(\omega)||_{L^\infty_Z}\right) $$
and so converges toward $0$ when $r$ goes to $\infty$.
\end{proof}

\section{Scattering}

\subsection{Penrose transformed free evolution}

We ware going to show that a solution $f$ of the non linear wave equation with initial data in $\Pi$ tends when $t\rightarrow \infty$ towards a solution of the free evolution with different initial data.

\smallskip

For that, we will resume at first to the Penrose transformed equation with no non linearity. Scattering on $f$ is different from the dynamics of $u$ on $T=\pi$ since $t=\infty$ doesn't correspond to $T=\pi$ but to $\lbrace T,R | \; T=\pi -R \rbrace$, which would be something to consider.

\begin{definition} Thanks to lemma \reff{free}, we note $L(t) f_\infty = f(t,.)$ the free evolution at time $t$ with initial data $f_\infty$. $L(t)f_\infty(r) = \Omega(T,R) \Re S(T)u_\infty (R)$.\end{definition}

\begin{lemma} Time $t=\infty$ corresponds by the Penrose transform to $\lbrace T=\pi-R \rbrace$.\end{lemma}

\begin{proof} We have that $t=\frac{\sin T}{\Omega}$, so $\sin T$ must be positive that is to say, $T\geq 0$ and $\Omega = 0$, ie $\cos T = -\cos R$, $T=\pi-R$, since $0<R<\pi$. \end{proof}

\subsection{Scattering result}

\begin{definition}
For any $u_0\in \Sigma$ we set : 

\begin{itemize}
\item $u(T,R)$ the solution of the non linear equation with initial data $u_0$,
\item $u_\infty = u_0 - i\int_{0}^\pi d\tau S(-\tau)H^{-1} F(\tau,u)$, let us remind that $F(\tau,u)=\Omega^{\alpha-2}(\tau)|\Re u|^\alpha \Re u$,
\item $S(T)u_\infty$ the free evolution with initial data $u_\infty$,
\item $(f_0,f_1)$ the PT transform of $u_0$,
\item $f=\Omega \Re u$ the solution of the non linear equation with initial data $(f_0,f_1)$,
\item $f_\infty $ the PT transform of $u_\infty$,
\item $L(t)f_\infty$ the free evolution with initial data $f_\infty$.
\end{itemize}
\end{definition}

\begin{proposition}
Let $p\in ]2\alpha,6[$ and $p>\Frac{16}{3}$, $s=\Frac{3}{2}-\Frac{4}{p}$, and $q$ such that $\Frac{3}{q} = \Frac{3}{2}-s$, ie $q=\Frac{3}{4}p>4$ and we have the Sobolev embedding $H^s\rightarrow L^q$. The function $f-L(t)f_\infty$ is in $L^q$ for all $t>0$ and its norm converges toward $0$ when $t\rightarrow \infty $.\end{proposition} 

\begin{proof} Let's compute $f-L(t)f_\infty$.

$$f-L(t)f_\infty = \Omega \Re (u-S(T)u_\infty)$$
and

$$u=S(T)u_0 -i\int_{0}^T d\tau S(T-\tau) H^{-1}F(\tau,u)(\tau, R) = S(T) u_\infty +i\int_{T}^\pi d\tau S(T-\tau)H^{-1}F(\tau,u)$$
so

$$f-L(t)f_\infty = \Omega \Re i\int_{T}^\pi d\tau S(T-\tau ) H^{-1}F(\tau,u)$$

We have now to compute the $L^q$ norm.

$$||f-L(t)f_\infty||_{L^q}^q \leq  \int r^2 dr |\Omega|^q |\int_{T}^\pi d\tau S(T-\tau) H^{-1}F(\tau,u)|^q$$

Let $R= \arct(t+r) -\arct(t-r)$, the change of variable gives

$$||f-L(t)f_\infty||_{L^q}^q \leq \int \sin^2 R dR \frac{\Omega^{q-4}}{2(t^2+r^2)} |\int_{T}^\pi d\tau S(T-\tau) H^{-1}F(\tau,u)|^q$$

Note that $T$ depends on $t$ and $R$ but not $\tau$.

$$||f-L(t)f_\infty||_{L^q}^q \leq \frac{1}{t^2} \int_{0}^\pi \sin^2 R dR \frac{\Omega^{q-4}}{2}|\int_{T}^\pi d\tau S(T-\tau) H^{-1}F(\tau,u)|^q$$

$$||f-L(t)f_\infty||_{L^q} \leq \frac{1}{t^{2/q}} ||\Omega^{(q-4)/q}|\int_{T}^\pi d\tau S(T-\tau) H^{-1}F(\tau,u)|\; ||_{L^q_R}$$

Since $\Omega \leq 2$, we have

$$||f-L(t)f_\infty||_{L^q} \leq \frac{C}{t^{2/q}} ||\int_{T}^\pi d\tau S(T-\tau) H^{-1}F(\tau,u) ||_{L^q_R}$$

$$||f-L(t)f_\infty||_{L^q} \leq \frac{C}{t^{2/q}} ||\int_{T}^\pi d\tau S(T-\tau) H^{-1}F(\tau,u) ||_{H^s_R}$$
by Sobolev embedding theorem.

$$||\int_{T}^\pi d\tau S(T-\tau) H^{-1}F(\tau,u) ||_{H^s_R} = ||H^{s-1}\int_{T}^\pi d\tau S(-\tau) F(\tau,u)||_{L^2_R}$$

$$= ||\int_{0}^\pi 1_{\tau>T} S(-\tau)F(\tau,u)||_{H^{s-1}}$$

$$\leq C||1_{\tau >T} F(\tau,u) ||_{Y^{1-s}_\pi}$$

$$\leq C ||F(\tau,u)||_{Y_\pi^{1-s}}$$

With $x'$ the conjugate number of $x$ such that $\frac{1}{x}+\frac{3}{x} = \frac{3}{2}-1+s$, that is $x'=\frac{4p}{2p+4}$, we get : 

$$||F(\tau,u)||_{Y_\pi^{1-s}}\leq C ||F(T,u)||_{L^{x'}_{T,R}} \leq C ||u||_{L^{(\alpha+1)x'}}^{\alpha+1}\leq C(||S(T)u_0||_{L^{(\alpha+1)x'}}^{\alpha+1}+||u-S(T)u_0||_{L^{(\alpha+1)x'}}^{\alpha+1}$$

But since $p>2\alpha$, we have $(\alpha+1)x' < p$, so

$$||f-L(t)f_\infty||_{L^q} \leq \frac{C}{t^{2/q}} (||S(T)u_0||_{L^p_{T,R}}^{\alpha+1}+||u-S(T)u_0||_{X^s_\pi}^{\alpha+1}) \leq \frac{C}{t^{2/q}}$$
since for $u_0\in \Sigma$, there exists $i$ such that $||S(T)u_0||_{L^p}\leq D\sqrt i$ and $||u-S(T)u_0||_{X^s_\pi}\leq C\sqrt i$.

\smallskip

Hence the result.

\smallskip

\end{proof}

\bibliographystyle{amsplain}
\bibliography{biblithde} 
\nocite{*}

\end{document}